\documentclass[10pt,a4paper]{article}

\usepackage{amsmath,amsthm,amsopn,amsfonts,amssymb,dsfont,color}
\usepackage{psfrag}
\usepackage{geometry}
\usepackage{graphicx}
\usepackage[utf8]{inputenc}
\usepackage[english]{babel}
\usepackage[active]{srcltx}

\usepackage{a4,a4wide}
\usepackage{color}
\usepackage{enumerate}

\def\ep{{\varepsilon}}
\def\R{\mathbb R}

\def \truc{\frac{1}{\left(A(1+\eta)\right)^{1/\eta}}}

\newtheorem{theo}{\textbf{Theorem}}[section]

\newtheorem{lem}[theo]{\textbf{Lemma}}
\newtheorem{prop}[theo]{\textbf{Proposition}}
\newtheorem{cla}[theo]{\textbf{Claim}}

\newtheorem{assumption}[theo]{\textbf{Assumption}}
\newtheorem{rem}[theo]{\textbf{Remark}}

\title{When fast diffusion and reactive growth both induce accelerating invasions}
\date{}

\begin{document}

\maketitle

\begin{center}
{\large\bf Matthieu Alfaro \footnote{IMAG, Univ. Montpellier, CNRS, Montpellier, France. E-mail: matthieu.alfaro@umontpellier.fr} and Thomas Giletti \footnote{IECL, Universit\'{e} de Lorraine, B.P. 70239, 54506
Vandoeuvre-l\`{e}s-Nancy Cedex, France.\\
 E-mail:
thomas.giletti@univ-lorraine.fr}.} \\
[2ex]
\end{center}




\vspace{10pt}

\begin{abstract} We focus on the spreading properties of solutions of monostable equations with fast diffusion. The nonlinear reaction term involves a weak Allee effect, which tends to slow down the propagation. We complete the picture of \cite{Alf-Gil-17} by studying the subtle case where acceleration does occur and is induced by a combination of fast diffusion and of reactive growth. This requires the construction of new elaborate sub and supersolutions thanks to some underlying self-similar solutions.
 \\


\end{abstract}

\maketitle

\section{Introduction} \label{s:intro}

In this paper, companion of \cite{Alf-Gil-17}, we are concerned with the {\it spreading properties}
of $u(t,x)$ the solution of the nonlinear monostable reaction-diffusion
equation
\begin{equation}\label{eq}
\partial _t u=\partial _{xx}(u^{m})+f(u),\quad t>0,\, x\in \R,
\end{equation}
in the {\it fast diffusion} regime $0<m<1$, the {\it linear diffusion} case $m=1$ and the  {\it porous medium diffusion} regime $m>1$ being already well understood. The typical nonlinearity $f$ we have in mind is $f(s)=rs^{\beta}(1-s)$, with $r>0$ and $\beta>1$ (Allee effect), the Fisher-KPP case $\beta=1$ being already well understood.  Equation \eqref{eq} is supplemented with a nonnegative initial data which is {\it front-like} and may have a {\it heavy
tail}, say $u_0(x)\lesssim  \frac 1{x^{\alpha}}$ for some $\alpha >0$ as $x\to+\infty$ (see Assumption~\ref{ass:initial}). 

Precisely, our goal is to understand the regime 
\begin{equation}
\label{range}
0<m<1, \quad \beta >1, \quad m+\frac{2}{\alpha}\leq \beta <2-m.
\end{equation}
In \cite{Alf-Gil-17}, we proved that, in this regime, propagation occurs by accelerating but a precise 
estimate of the position of the level sets of $u(t,\cdot)$ as $t\to +\infty$ was missing. In the present paper, we fill
this gap by constructing very refined sub and supersolutions, which rely on self-similar solutions of
\begin{equation}\label{eq-lin}
\partial _t u=\partial _{xx}(u^{m})+r u^{\beta}.
\end{equation}
In particular, and roughly speaking, we show that the leading term of the position of the level sets is of the monomial type $t^{\frac{\beta -m}{2(\beta-1)}}$, which is independent on $\alpha$, thus on the tail of the initial data. This is in contrast with the other regimes fully described in \cite{Alf-Gil-17}.

\medskip

We refer to the introduction in \cite{Alf-Gil-17} for references, comments and relevance in population dynamics models on the three main effects inserted in the Cauchy problem \eqref{eq}: nonlinear diffusion \cite{Kin-Mac-03}, \cite{Vaz-book2}, \cite{Aud-Vaz-17}, Allee effect \cite{Cou-08} (vs KPP nonlinearities \cite{Fis-37}, \cite{Kol-Pet-Pis-37}), tail of the initial data. Let us briefly recall the available results on the propagation of solutions to \eqref{eq}, with a front-like initial data (whose behaviour at $+\infty$ is of crucial importance on the speed of invasion).

$\bullet$ In the KPP situation ($\beta=1$): the linear diffusion ($m=1$) case was studied by Hamel and Roques \cite{Ham-Roq-10}, revealing in particular that an algebraic initial tail implies that the level sets of $u(t,\cdot)$ travel exponentially fast as $t\to+\infty$. See also \cite{Gar-11}, \cite{Cab-Roq-13}, \cite{Gui-Hua-preprint}, \cite{Alf-Cov-preprint}. The nonlinear diffusion cases ($m>1$ or $0<m<1$) were recently solved in \cite{Alf-Gil-17}, revealing similar results, except in the fast diffusion range $0<m<1-\frac{2}{\alpha}$ which yields a slightly  stronger acceleration.

$\bullet$ In presence of an Allee effect ($\beta>1$): the linear diffusion ($m=1$) case was studied in \cite{Alf-tails}. For algebraic tails, the exact
separation between acceleration or not (depending on the strength of the Allee effect) was obtained: when $\beta<1+\frac 1 \alpha$ acceleration occurs,  and the location of the level sets of the solution is of the monomial type $t^{\frac{1}{\alpha(\beta-1)}}$ as $t\to+\infty$. The nonlinear diffusion cases were recently studied in \cite{Alf-Gil-17}: in the porous medium diffusion case $m>1$, the obtained results were sharp and very similar to the case $m=1$. On the other hand, because of the possible acceleration induced by fast diffusion itself, the case $0<m<1$ is much more subtle. We proved in~\cite{Alf-Gil-17} that acceleration occurs if and only if $\beta<\max\left(1+\frac 1 \alpha,2-m\right)$. Next, in the range
\begin{equation*}
\label{range-deja-ok}
\beta<\min\left(1+\frac 1\alpha,m+\frac 2 \alpha\right),
\end{equation*}
we precisely estimated the position of the level sets, again of the monomial type $t^{\frac{1}{\alpha(\beta-1)}}$. The keystone for constructing accurate sub and supersolutions was  the solution $w(t,x)$ of the ODE Cauchy problem ($x$ playing the role of a parameter)
\begin{equation}
\label{edo-w-1}
\partial _t w(t,x)=r w^\beta (t,x),\quad w(0,x)=u_0(x).
\end{equation}
On the other hand, in the remaining parameter range, which rewrites as \eqref{range}, the acceleration is induced by a combination of fast diffusion and of reactive growth. Since \eqref{edo-w-1} neglects the former, it was not enough to precisely quantify the acceleration phenomena. The main novelty of the present paper is to use a self-similar solution of \eqref{eq-lin} to build accurate sub and supersolutions, which enable to understand the acceleration regime \eqref{range} that was missing in \cite{Alf-Gil-17}.

\medskip

Through this work, we make the following assumption on the initial condition. 

\begin{assumption}[Initial condition]\label{ass:initial} The initial condition $u_0:\R\to \left[0,1\right]$ is uniformly continuous and asymptotically front-like, in the sense that
\begin{equation*}
\label{front-like}
\liminf _{x\to -\infty} u_0(x)>0,
\end{equation*}
and
\begin{equation}
\label{algebraic-acc-FDE}
 u_0(x)\leq \frac{\overline C}{x^\alpha},\quad \forall x\geq x_0,
\end{equation}
for some $\alpha>0$, $\overline C>0$ and $x_0>1$.
\end{assumption}

Notice that no lower bound is required in \eqref{algebraic-acc-FDE}. In particular $u_0\equiv 0$ on $(x_0,+\infty)$ is allowed. The reason is that the fast diffusion equation makes the tail of the solution (at least) algebraically heavy at any positive time: at time $T>0$, there is $C(T)>0$ such that  $u(T,x)\geq \frac{C(T)}{x^{\frac{2}{1-m}}}$ for $x$ large enough, and moreover $C(T)\to+\infty$ as $T\to +\infty$. This was proved by  Herrero and Pierre~\cite{Her-Pie-85}, and will be used in subsection \ref{ss:lower-kpp}.

As far as the nonlinearity $f$ is concerned, we assume the following.

\begin{assumption}[Monostable nonlinearity with Allee effect]\label{ass:f} The nonlinearity $f:\left[0,1 \right]\to \R$ is of the class~$C^{1}$, and is of the monostable type, in the sense that
$$
f(0)=f(1)=0, \quad f>0 \; \text{ in } (0,1).
$$
Moreover there are $\beta>1$, $r>0$, $\overline r>0$, and $s_0\in(0,1)$ such that
\begin{equation} 
\label{nonlinearity-acc-FDE}
f(s)\geq r s^\beta  ,\quad \forall 0\leq s\leq s_0,
\end{equation}
and
\begin{equation}
\label{nonlinearity-accbis-FDE}
f(s)\leq \overline r s^\beta,\quad \forall 0\leq s\leq 1.
\end{equation}
\end{assumption}

In the sequel, we always denote by $u(t,x)$ the solution of \eqref{eq} with initial condition $u_0$. From the above assumptions and the comparison principle, one gets $0\leq u(t,x)\leq 1$ and even
\begin{equation*}\label{strict}
0<u(t,x)<1, \quad \forall (t,x)\in(0,+\infty)\times\R.
\end{equation*}
Also, since the initial data is front-like, the state $u\equiv 1$ does invade the whole line $\R$ as $t\to+\infty$: there is $c _0>0$ such that 
\begin{equation}
\label{invasion}
\lim _{t\to +\infty} \inf _{x\leq c _0 t} u(t,x)=1,
\end{equation}
meaning  that propagation is at least linear. We also have 
\begin{equation}
\label{zero-a-droite}
\lim _{x\to+\infty}u(t,x)=0,\quad \forall t\geq 0.
\end{equation}
We refer to \cite{Alf-Gil-17} for proofs of those basic facts. 

In order to state our results we define, for any $\lambda \in (0,1)$ and $t\geq 0$,
$$
E_\lambda(t):=\{x\in\R:\, u(t,x)=\lambda\}
$$
the $\lambda$-level set of $u(t,\cdot)$. In view of \eqref{invasion} and \eqref{zero-a-droite}, for any $\lambda\in(0,1)$, there is a time $t_\lambda >0$ such that
\begin{equation*}
\label{nonvide}
\emptyset \neq E_\lambda(t) \subset (c _0 t, +\infty), \quad \forall t\geq t_\lambda.
\end{equation*}
Our main result is the following.

\begin{theo}[Localization of the accelerating level sets in the range \eqref{range}]\label{th:main} Let $m \in (0, 1)$, $\alpha >\frac{1}{1-m}$ (from Assumption \ref{ass:initial}),  and $\beta>1$ (from Assumption \ref{ass:f}) be such that
\begin{equation}\label{alpha-beta-acc-FDE}
m+\frac{2}{\alpha}\leq \beta <2-m.
\end{equation}
Then, for any $\lambda \in(0,1)$, any small $\ep>0$,  there is a
time $T_{\lambda,\ep}\geq t_\lambda$ such that
\begin{equation}
\label{levelset-FDE}
E_\lambda(t)\subset (x^-(t),x^{+}(t)),\quad \forall t\geq  T _{\lambda,\ep},
\end{equation}
where 
\begin{equation}
\label{x-m-p}
 x^-(t):=(1-\ep)z_0 t^{\frac{\beta-m}{2(\beta-1)}}, \quad x^+(t):=\left(\overline r\overline C^{\beta-1}(\beta-1)t\right)^{\frac{\beta-m+\ep }{2(\beta-1)}},
 \end{equation}
 where $z_0>0$ is a constant depending on $m$, $\beta$, $r$ and $\ep$ (the initial data is irrelevant for the lower bound).

 If we assume furthermore that $\alpha\geq \frac{2}{1-m}$ then the upper bound is sharply improved to 
 \begin{equation}
 \label{improved}
 x^+(t):=(1+\ep)\overline{z_0} t^{\frac{\beta-m}{2(\beta-1)}},
 \end{equation}
 where $\overline{z_0}>0$ is a constant depending on $m$, $\beta$ and $\overline r$ (but not on $\ep$), and not on $\alpha$ nor $\overline C$ (the initial data is irrelevant for the upper bound).
\end{theo}
\begin{rem}\label{rem:up1}
As a matter of fact, the upper estimate \eqref{improved} can be extended (up to changing the multiplicative constant) to the regime $m+\frac 2 \alpha \leq \beta <2-m$, $\frac{1}{1-m}<\alpha<\frac{2}{1-m}$. The proof is rather similar and simply relies on a different choice of a self-similar solution with a slower decay at infinity. To make the presentation simpler, we treat this case separately and only sketch the necessary changes in Remarks~\ref{rem:up2} and~\ref{rem:up3}. We point out that the multiplicative constant in \eqref{improved} may still be chosen independently of the initial data as long as $m+\frac 2\alpha < \beta$.
\end{rem}

\medskip

The paper is organized as follows. In Section \ref{s:self}, we state the existence of an adequate self-similar solution of equation \eqref{eq-lin}, which is the main tool for construction of sharp sub and supersolutions. The latter is achieved in Section \ref{s:proof}, thus proving Theorem \ref{th:main}. Last, the actual construction of the self-similar solution is performed in Section \ref{s:construction}.

\section{A key self-similar solution}\label{s:self}

Guided by \cite{Kin-Mac-03}, we plug the self-similar ansatz
\begin{equation}
\label{def:self}
w(t,x):=\frac{1}{t^{\frac{1}{\beta-1}}}\varphi\left(\frac{x}{t^{\frac{\beta-m}{2(\beta-1)}}}\right)
\end{equation}
into 
\begin{equation}
\label{eq-lin1}
\partial _t u =\partial _{xx} (u^m)+ru^\beta,
\end{equation}
and obtain after some straightforward computations that one needs
\begin{equation}
\label{eq-self-similar}
-\frac{1}{\beta-1}\left(\varphi(z)+\frac{\beta-m}{2}z\varphi'(z)\right)=(\varphi ^m)''(z)+r\varphi^\beta(z).
\end{equation}
In this section, we claim the existence of such a self-similar solution having the required asymptotics properties for our analysis to work in Section \ref{s:proof}. The proofs are postponed to Section \ref{s:construction}.
\medskip 

In the range \eqref{alpha-beta-acc-FDE} of Theorem \ref{th:main}, there are $z_0>0$, $C_0>0$, $C_\infty>0$ such that the following holds: for any $k_0 <C_0< K_0$, any $0<k_\infty <C_\infty$ and some $K_\infty >C_\infty$, there is a decreasing function $\varphi:(z_0,+\infty)\to(0,+\infty)$ solving \eqref{eq-self-similar} on $(z_0,+\infty)$ and satisfying the following boundary estimates
\begin{equation}
\label{self-z0}
\frac{k_0}{(z-z_0)^{\frac{1}{\beta-1}}} \leq \varphi(z) \leq  \frac{K_0}{(z-z_0)^{\frac{1}{\beta-1}}},\quad \text{ in a right neighborhood of   $z_0$, say $(z_0,z_1)$},
\end{equation}
and
\begin{equation}
\label{self-infini}
\frac{k_\infty}{z^{\frac{2}{1-m}}}\leq \varphi(z) \leq  \frac{K_\infty}{z^{\frac{2}{1-m}}},\quad \text{ in a neighborhood of $+\infty$}.
\end{equation}

Hence we are equipped with \eqref{def:self} solving \eqref{eq-lin1} in the domain $t>0$, $x>z_0 t^{\frac{\beta-m}{2(\beta-1)}}$. For later use, we use the convention
\begin{equation}
\label{convention}
w(t,x)=+\infty, \quad t>0,\, x\leq z_0 t^{\frac{\beta-m}{2(\beta-1)}}.
\end{equation}

\medskip

\noindent {\bf The blow-up zone.}  We will also need the following estimates: up to reducing $z_1$, we have that
\begin{equation}\label{to-be-proved-2abis}
- \frac{k_0^{1-\beta}}{\beta-1} \varphi^\beta (z) < \varphi '(z)  < - \frac{K_0^{1-\beta}}{\beta-1} \varphi^\beta (z)\quad \text{ in } (z_0,z_1).
\end{equation}
Also, for any $\delta >0$, if $\max(K_0-C_0,C_0-k_0)$ is sufficiently small then, up to reducing $z_1$, we have that
\begin{equation}\label{to-be-proved-2b} \vert\varphi''(z)\vert \leq \delta  \varphi^{\beta +1 -m}(z) \quad \text{ in } (z_0,z_1).
\end{equation}
The two above estimates will also be proved in Section \ref{s:construction}.

Going back to $w(t,x)$ this transfers into the following: there is $\gamma >0$ so that
\begin{equation}
\label{self-infini-5}
 - \partial_x w (t,x) =  \vert \partial _x w(t,x)\vert \leq \frac{\gamma}{t^{\frac{2-m-\beta}{2(\beta-1)}}}w^{\beta}(t,x), \quad  \vert \partial_{xx} w (t,x)\vert \leq  \delta  w^{\beta+1-m}(t,x),
 \end{equation}
 and
 \begin{equation}
 \label{self-infini-6}\vert \partial _x (w^{m})(t,x)\vert \leq \frac{\gamma}{t^{\frac{2-m-\beta}{2(\beta-1)}}}w^{\beta+m-1}(t,x), \quad \vert \partial _{xx} (w^{m})(t,x)\vert \leq \frac{\gamma}{t^{\frac{2-m - \beta}{\beta-1}}} w^{2\beta + m - 2} + \delta   w^{\beta}(t,x),
\end{equation}
for all  $t>0$ and  $z_0t^{\frac{\beta-m}{2(\beta-1)}} < x < z_1 t^{\frac{\beta-m}{2(\beta-1)}}$.

\medskip

\noindent {\bf The far away zone, or tail.}  From \eqref{self-infini}, we have that
\begin{equation}
\label{self-infini-2}
 \frac{ k_\infty \, t^{\frac{1}{1-m}}}{x^{\frac{2}{1-m}}}\leq w(t,x)\leq \frac{K_\infty\, t^{\frac{1}{1-m}}}{x^{\frac{2}{1-m}}}, 
\end{equation}
for all $t>0$ and  $x  t^{- \frac{\beta-m}{2(\beta-1)}}$ large enough.

Again, we will also need estimates on the derivatives of $\varphi $: there are $K'>0$ and $K''>0$ such that, for $z$ large enough,
\begin{equation}
\label{to-be-proved}
-K'\varphi ^{\frac{3-m}{2}}(z)\leq \varphi '(z)<0,\quad \vert \varphi''(z)\vert \leq K'' \varphi ^{2-m}(z),
\end{equation}
which will also be proved in Section \ref{s:construction} (notice that the estimate on the second derivative actually simply follows from the ODE and previous estimates on $\varphi$ and $\varphi '$ as $z \to +\infty$). Due to the positivity of $\varphi$ and $-\varphi'$, up to changing the positive constants $K'$ and $K''$, we can also assume that \eqref{to-be-proved} holds in the whole interval $[z_1, + \infty)$.

Going back to $w(t,x)$ this transfers into the following: up to enlarging $\gamma >0$,
\begin{equation}
\label{self-infini-3}
 -\partial _x w(t,x)=\vert \partial _x w(t,x)\vert \leq \frac{\gamma}{t^{\frac 12}}w^{\frac{3-m}{2}}(t,x), \quad  \vert \partial_{xx} w (t,x)\vert \leq \frac{\gamma}{t}w^{2-m}(t,x),
 \end{equation}
 and
 \begin{equation}
 \label{self-infini-4}
 \vert \partial _x (w^{m})(t,x)\vert \leq \frac{\gamma}{t^{\frac 12}}w^{\frac{m+1}{2}}(t,x), \quad  \vert \partial _{xx} (w^{m})(t,x)\vert \leq \frac{\gamma}{t}w(t,x),
\end{equation}
for all  $t>0$ and  $x \ge z_1 t^{\frac{\beta-m}{2(\beta-1)}}$.

\section{Proof of the main result}\label{s:proof}

\subsection{Lower bound on the level sets in  \eqref{levelset-FDE}}\label{ss:lower-kpp}

Thanks to the self-similar solution of Section \ref{s:self}, we are in the position to construct an accurate subsolution to \eqref{eq}, whose level sets travel like $x^{-}(t)$ appearing in \eqref{x-m-p}. We start with some preparation. 

\medskip

Let $\ep >0$ small be given.  We fix a large enough $\eta>0$ so that
\begin{equation}
\label{def:eta}
\eta  > 1,\quad \frac{r\beta}{1+\eta}<r-\ep,
\end{equation}
and a large enough  $A>0$  so that
 \begin{equation}
 \label{def:A}
\frac{1}{\left(A(1+\eta)\right)^{1/\eta}} \leq s_0,
\end{equation}
where $s_0>0$ is as in \eqref{nonlinearity-acc-FDE}. We also select $\delta>0$ small enough so that
\begin{equation}
\label{eq:delta}
A(3+3\eta)\delta \leq \frac \ep 2.
\end{equation}

We define
\begin{equation}
\label{def:self-2}
w(t,x):=\frac{1}{t^{\frac{1}{\beta-1}}}\varphi\left(\frac{x}{t^{\frac{\beta-m}{2(\beta-1)}}}\right)
\end{equation}
 as the self-similar solution \eqref{def:self} of Section \ref{s:self}, solving \eqref{eq-lin1} with $r\leftarrow r-\ep$. Moreover, we choose the constants $k_0$ and $K_0$ in \eqref{self-z0} close enough to $C_0$ so that \eqref{to-be-proved-2b} holds, and so do all estimates of Section \ref{s:self}.

Now, for any $t >0$, from $\varphi(z_0)=+\infty$, $\varphi(+\infty)=0$ and the monotonicity of $\varphi$, we can define
 \begin{eqnarray*}
X(t)&:=&\sup\left\{x>z_0 t^{\frac{\beta-m}{2(\beta-1)}}: w(t,x)= \frac{1}{\left(A(1+\eta)\right)^{1/\eta}} \right\}\\
&=& \sup\left\{x>z_0 t^{\frac{\beta-m}{2(\beta-1)}}: \varphi \left(\frac{x}{t^{\frac{\beta-m}{2(\beta-1)}}}\right)=\frac{t^{\frac{1}{\beta-1}}}{\left(A(1+\eta)\right)^{1/\eta}} \right\},
\end{eqnarray*}
so that

\begin{equation}
\label{consequence}
w(t,X(t)) =\truc, \quad \text{ and }\, w(t,x) < \truc\; \text{ for all }  x > X(t).
\end{equation}
Notice that from expression \eqref{def:self-2} we clearly have $X(t) < z_1t^{\frac{\beta-m}{2(\beta-1)}}$ for $t>0$ large enough.

\begin{lem}[An accelerating subsolution]
\label{lem:sub-kpp} Let the assumptions of Theorem~\ref{th:main} hold.  Define
\begin{equation*}
\label{def-sub-kpp}
v(t,x):=
\begin{cases}
\truc\frac{\eta}{1+\eta} &\mbox{ if } x\leq X(t)\vspace{5pt}\\
w(t,x)\left(1-A w^\eta (t,x)\right) &\mbox{ if } x>X(t).
\end{cases}
\end{equation*}
Then  there is $T>0$ large enough so that $v(t,x)$ is a subsolution to \eqref{eq} in the domain $(T,+\infty)\times \R$.
\end{lem}

\begin{proof} Let us note that $v$ is smooth in both subdomains $\{ x < X(t)\}$ and $\{ x > X(t) \}$. Also, it is continuous in $(0,+\infty) \times \mathbb{R}$ as well as $C^1$ with respect to $x$ at the junction point $X(t)$. This means that a comparison principle is applicable provided that $v$ satisfies 
\begin{equation}\label{plus1}
\mathcal L v(t,x):=\partial _t v(t,x)-\partial _{xx}(v^m)(t,x)-f(v(t,x))\leq 0
\end{equation}
in both these subdomains
. Since $\truc \frac{\eta}{1+\eta}$ is obviously a subsolution to \eqref{eq}, we only need to check this inequality when $t>T$, $x > X(t)$. First it is straightforward that $v(t,x)\leq \max_{w \geq 0} w (1- Aw^\eta) = \truc \frac{\eta}{1+\eta}\leq s_0$ in view of \eqref{def:A}. It then follows from \eqref{nonlinearity-acc-FDE} and a  convexity inequality that
\begin{equation}
\label{un}
f(v)\geq rw^\beta -rA\beta w^{\beta + \eta }.
\end{equation}
Next we have
\begin{equation}
\label{deux}
\partial _t v=(\partial _t w)\left(1-A(1+\eta)w^\eta\right)=\left(\partial _{xx}(w^m)+(r-\ep) w^\beta\right)\left(1-A(1+\eta)w^\eta\right).
\end{equation}
Also, we have
\begin{eqnarray}
\partial _{xx}(v^{m})&=&\partial _{xx}(w^{m})(1-Aw^{\eta})^{m}+2\partial _x(w^{m})\partial _x((1-Aw^{\eta})^{m})+w^{m}\partial_{xx}((1-Aw^{\eta})^{m})\nonumber \\
&= & \partial _{xx}(w^{m})(1-Aw^{\eta})^{m}-2Am\eta\partial_x(w^{m})(\partial _x w) w^{\eta-1}(1-Aw^{\eta})^{m-1}\nonumber\\
&& -Am\eta(\eta-1)(\partial _x w)^{2}w^{m+\eta-2}(1-Aw^{\eta})^{m-1}\nonumber\\
&& -Am\eta (\partial _{xx} w ) w^{m+\eta-1} (1-Aw^{\eta})^{m-1}\nonumber\\
& & - (1-m) A^{2} m \eta^2 (\partial _x w)^{2} w^{m+2 \eta -2} (1-Aw^\eta)^{m-2}.\label{trois}
\end{eqnarray}
Plugging \eqref{un}, \eqref{deux} and \eqref{trois} into \eqref{plus1} we arrive at
\begin{eqnarray*}
\mathcal L v&=&-\ep w^{\beta} +w^{\beta+\eta}A(r\beta-(r-\ep)(1+\eta))\\
&& +\partial _{xx}(w^{m})\left[1-A(1+\eta)w^{\eta}-(1-Aw^{\eta})^{m}\right]\\
&& \nonumber + 2Am\eta\partial_x(w^{m})(\partial _x w) w^{\eta-1}(1-Aw^{\eta})^{m-1}\\
&& +Am\eta(\eta-1)(\partial _x w)^{2}w^{m+\eta-2}(1-Aw^{\eta})^{m-1}\nonumber\\
&& +Am\eta (\partial _{xx} w )w^{m+\eta-1}(1-Aw^{\eta})^{m-1}\nonumber\\
& & +(1-m) A^{2} m \eta^2 (\partial _x w)^{2} w^{m+2 \eta -2}  (1-Aw^\eta)^{m-2}.
\end{eqnarray*}
The second term in the above right hand side member is nonpositive in view of \eqref{def:eta}. We deduce from \eqref{consequence} that
 $1 \geq 1 - A w^\eta (t,x) \geq \frac{\eta}{1+\eta} \geq \frac{1}{2}$, so that $(1-Aw^{\eta})^{m-1}\leq 2^{1-m}$, $(1-Aw^{\eta})^{m-2}\leq 2^{2-m}$. By the mean value theorem $\vert 1-(1-Aw^{\eta})^{m}\vert \leq 2^{1-m}Am  w^{\eta}$. Hence, 
\begin{eqnarray}
\mathcal L v&\leq &-\ep w^{\beta} \nonumber\\
&& + |\partial _{xx}(w^{m})| A\left(1+\eta+2^{1-m}m\right)w^{\eta}\nonumber\\
&& \nonumber + 2^{2-m}Am\eta\partial_x(w^{m})(\partial _x w) w^{\eta-1}\\
&& +2^{1-m}Am\eta(\eta-1)(\partial _x w)^{2}w^{m+\eta-2}\nonumber\\
&& +2^{1-m}Am\eta \vert \partial _{xx} w \vert w^{m+\eta-1}\nonumber\\
& & +2^{2-m}(1-m) A^2 m \eta^2 (\partial _x w)^{2} w^{m+2 \eta -2} .\label{infer}
\end{eqnarray}
Now we need to estimate the derivatives of $w$ by powers of $w$. To do so, we distinguish two regions.

\noindent {\bf The far away zone $t>0$, $x\geq z_1 t^{\frac{\beta-m}{2(\beta-1)}}$,} where we take advantage of \eqref{self-infini-3} and \eqref{self-infini-4} to infer from~\eqref{infer} that
\begin{eqnarray*}
\mathcal L v&\leq & -\ep w^{\beta} \\
&& +A\left(1+\eta+2^{1-m}m\right)\frac \gamma t w^{\eta+1}\\
&& \nonumber + 2^{2-m}Am\eta\frac{\gamma ^{2}}{t}w^{\eta+1}\\
&& +2^{1-m}Am\eta(\eta-1)\frac{\gamma ^{2}}{t}w^{\eta +1}
\nonumber\\
&& +2^{1-m}Am\eta \frac{\gamma}{t}w^{\eta +1}\nonumber\\
& & +2^{2-m}(1-m) A^2 m \eta^2 \frac{\gamma ^{2}}{t} w^{2 \eta +1}.
\end{eqnarray*}
Hence, for a positive constant $K=K(\eta,m,\gamma)$, we have 
$$
\mathcal L v\leq -\ep w^{\beta}+AK \frac{w^{\eta+1}}t+ A^{2}K \frac{w^{2\eta+1}}t\leq  w^\beta\left(-\ep +AK \frac{1}{t}+A^2 K\frac 1t\right)\,
$$
from $\beta<2-m<2<\eta+1 <2\eta+1$ and the crude estimate $w<1$. Hence for $T=T(\eta,m,\gamma,A,\ep)>0$ large enough we have $\mathcal L v(t,x)\leq 0$ for $t> T$, $x\geq z_1 t^{\frac{\beta-m} {2(\beta-1)}}$, that is after a large time and in the far away zone.

\noindent {\bf The blow-up zone $t>0$, $z_0t^{\frac{\beta-m}{	2(\beta-1)}}< X(t) < x< z_1 t^{\frac{\beta-m}{2(\beta-1)}}$,} where we take advantage of \eqref{self-infini-5} and \eqref{self-infini-6} to infer from \eqref{infer} that
\begin{eqnarray*}
\mathcal L v&\leq & -\ep w^{\beta} \\
&& +A\left(1+\eta+2^{1-m}m\right) \left( \delta  w^{\beta+\eta} +  \frac{\gamma}{t^{\frac{2-m - \beta}{\beta-1}}} w^{ 2 \beta +\eta + m - 2}\right)\\
&& \nonumber + 2^{2-m}Am\eta\frac{\gamma ^{2}}{t^{\frac{2-m-\beta}{\beta-1}}}w^{2\beta+\eta + m-2}\\
&& +2^{1-m}Am\eta(\eta-1)\frac{\gamma ^{2}}{t^{\frac{2-m-\beta}{\beta-1}}}w^{2\beta+\eta +m-2}
\nonumber\\
&& +2^{1-m}Am\eta \delta   w^{\beta+\eta} \nonumber\\
& & +2^{2-m}(1-m) A^2 m \eta^2 \frac{\gamma ^{2}}{t^{\frac{2-m-\beta}{\beta-1}}} w^{2\beta+2\eta +m-2}.
\end{eqnarray*}
Hence, for a positive constant still denoted $K=K(\eta,m,\gamma)$, we have 
$$
\mathcal L v\leq w^\beta\left(-\ep + A (3+3\eta ) \delta   + AK \frac{1}{t^{\frac{2-m-\beta}{\beta-1}}}+A^2K\frac 1{t^{\frac{2-m-\beta}{\beta-1}}}\right),
$$
since $\beta+\eta+m-2>0$ and $w<1$. Recall \eqref{eq:delta} and that $2-m-\beta>0$. Hence, up to enlarging  $T=T(\eta,m,\gamma,A,\ep)>0$, we have $\mathcal L v(t,x)\leq 0$ for $t> T$, $X(t)< x< z_1 t^{\frac{\beta-m} {2(\beta-1)}}$, that is after a large time and in the blow-up zone.
\end{proof}

\begin{lem}[Comparison modulo shifts]
\label{lem:initial}
Let the assumptions of Theorem \ref{th:main} hold. Let $v(t,x)$ be the subsolution to \eqref{eq} on $(T,+\infty)\times \R$, as defined in Lemma \ref{lem:sub-kpp}. Then there are a large enough time $T'>T$ and a large enough shift $X'>0$ such that
\begin{equation}
\label{initial}
v(T,x)\leq u(T',x-X'),\quad \forall x\in \R.
\end{equation}
\end{lem}

\begin{proof} By the comparison principle, we can assume without loss of generality  that $u_0$ is nonincreasing on $\R$, and therefore, from the comparison principle again,
\begin{equation}
\label{decroit}
\text{ for any $t>0$, $u(t,\cdot)$ is nonincreasing on $\R$.}
\end{equation}
Moreover, we have from \eqref{self-infini-3} that
\begin{equation}
\label{decroit-bis}
\text{ for any $t>0$, $w(t,\cdot)$ is nonincreasing on $(z_1t^{\frac{\beta-m}{2(\beta-1)}},+\infty)$.}
\end{equation}
Recall also that, up to enlarging $T>0$ from Lemma \ref{lem:sub-kpp}, we have 
$X(T)\leq z_1T^{\frac{\beta-m}{2(\beta-1)}}$. Now observe that, thanks to \eqref{self-infini},
$$
w(T,x)=\frac 1{T^{\frac{1}{\beta-1}}}\varphi\left(\frac{x}{T^{\frac{\beta-m}{2(\beta-1}}}\right)\leq  \frac{K_\infty T^{\frac{1}{1-m}}}{x^{\frac{2}{1-m}}}, \quad \text{ for any $x$ large enough. }
$$
Hence there is $X_1=X_1(T)>z_1T^{\frac{\beta-m}{2(\beta-1)}}\geq X(T)$ such that
\begin{equation}
\label{v(T)}
v(T,x)\leq \begin{cases} \truc \frac{\eta}{1+\eta} & \text{ for all } x\in \R \vspace{5pt}\\
w(T,x)\leq \frac{K_\infty T^{\frac{1}{1-m}}}{x^{\frac{2}{1-m}}} &\text{ for } x\geq X_1.
\end{cases}
\end{equation}
Next, from the invasion result \eqref{invasion}, there is $T_0>T$ such that 
\begin{equation}
\label{next}
u(t,x)\geq \truc\frac \eta{1+\eta}, \quad \forall t\geq T_0, \forall x\leq X_1.
\end{equation}

Last, by comparison with the fast diffusion equation ---namely $\partial _t u\geq \partial _{xx}(u^m)$--- and thanks to  \cite[Theorem 2.4]{Her-Pie-85}, we know that for any $T'>0$, there is $C(T')>0$ such that $u(T',x)\geq \frac{C(T')}{x^{\frac{2}{1-m}}}$ for $x$ large enough, and moreover $C(T')\to +\infty$ as $T'\to +\infty$. Hence, we can now fix $T'>T_0$ so that, for some $X_2=X_2(T')>X_1$,
\begin{equation}
\label{last}
u(T',x)\geq \frac{K_\infty  T^{\frac{1}{1-m}}}{x^{\frac{2}{1-m}}}, \quad \forall  x\geq X_2.
\end{equation}

Now we define $X':=X_2-X_1$ and prove \eqref{initial}, by dividing into three regions. When $x\leq X_2$, so that $x-X' \leq X_1$, this follows from \eqref{next} and the upper line in \eqref{v(T)}. When $X_2\leq x\leq 2X_2-X_1$, so that $X_1\leq x-X'\leq X_2$, we successively use \eqref{decroit}, \eqref{last}, the second line in \eqref{v(T)} and \eqref{decroit-bis} to get
$$
u(T',x-X')\geq u(T',X_2)\geq w(T,X_2)\geq w(T,x)\geq v(T,x).
$$
When $x\geq 2X_2-X_1$, so that $x-X'\geq X_2$, we successively use \eqref{last} and the second line in \eqref{v(T)} to obtain
$$
u(T',x-X')\geq \frac{K_\infty T^{\frac{1}{1-m}}}{(x-X')^{\frac{2}{1-m}}}\geq \frac{K_\infty T^{\frac{1}{1-m}}}{x^{\frac{2}{1-m}}}\geq w(T,x)\geq v(T,x).
$$
This completes the proof of \eqref{initial}.
\end{proof}

From Lemma \ref{lem:sub-kpp} and Lemma \ref{lem:initial}, we deduce that
$$
v(t+T-T',x+X')\leq u(t,x), \quad \forall (t,x)\in(T',+\infty)\times \R.
$$
Now, the proof is the same as that in \cite{Alf-tails} or \cite[subsections 5.1 and 6.2]{Alf-Gil-17}. Roughly speaking, the  subsolution \lq\lq lifts'' the solution $u(t,x)$ on intervals that enlarge with the correct acceleration, which provides the lower bound in \eqref{x-m-p} on the level sets $E_\lambda(t)$ when $\lambda$ is small. Next, the estimate for larger $\lambda$ is obtained thanks to the fact that invasion occurs for front-like initial data, see \eqref{invasion}. We omit the details and conclude that the lower bound in \eqref{x-m-p} is proved. \qed

\subsection{Upper  bound on the level sets  in  \eqref{levelset-FDE}}\label{ss:upper-kpp}

Let $\lambda \in(0,1)$ and $\ep >0$ small be given. The expression of $x^+(t)$ in \eqref{x-m-p} was already proved in \cite{Alf-Gil-17}. We thus assume  $\alpha\geq \frac{2}{1-m}$ and look after the improvement \eqref{improved}. Again the self similar solution of Section \ref{s:self} provides a more accurate supersolution.

 In view of  $\alpha\geq \frac{2}{1-m}$, the upper bound in \eqref{algebraic-acc-FDE},  and the comparison principle, it is enough (to prove the upper estimate on the level sets) to consider the case where
\begin{equation}
\label{algebraic-encore}
u_0(x)=\frac{\overline C}{x^{\frac{2}{1-m}}}, \quad \forall x\geq x_0>1.
\end{equation}

We define $w(t,x)$ as the self-similar solution \eqref{def:self} of Section \ref{s:self}, solving \eqref{eq-lin1} with $r\leftarrow \overline r$ (so that $z_0\leftarrow \overline{z_0}$, $C_\infty\leftarrow \overline{C_\infty}$ etc.).
From \eqref{nonlinearity-accbis-FDE} it is immediate that
$$
\psi(t,x):=\min\left(1,w(t,x)\right)
$$
is a supersolution for equation \eqref{eq}. Recalling convention \eqref{convention}, notice that $\psi (t,x)=1$ in the domain $t>0$, $x\leq \overline{z_0} t^{\frac{\beta-m}{2(\beta-1)}}$.

Now selecting $T>0$ large enough so that
\begin{equation}
\label{T}
k_\infty T^{\frac{1}{1-m}}\geq \overline C,\quad  x_0\leq \overline{z_0} T^{\frac{\beta-m}{2(\beta-1)}},\quad \frac{1}{T^{\frac{1}{\beta-1}}}\min_{\overline{z_0}<z\leq \overline{z_1}}\varphi(z)\geq \frac{\overline C}{\overline{z_0}^{\frac{2}{1-m}}T^{\frac{\beta-m}{(1-m)(\beta-1)}}},
\end{equation}
we claim that $\psi(T,\cdot)\geq u_0$. Indeed in the range $x\leq \overline{z_0} T^{\frac{\beta-m}{2(\beta-1)}}$ this is clear; in the range $\overline{z_0} T^{\frac{\beta-m}{2(\beta-1)}}<x<\overline{z_1}T^{\frac{\beta-m}{2(\beta-1)}}$, which enforces $x>x_0$ in view of the second inequality in \eqref{T}, this is a consequence of \eqref{algebraic-encore} and the third inequality in \eqref{T}; last, in the range $x\geq \overline{z_1}T^{\frac{\beta-m}{2(\beta-1)}}$ this follows from \eqref{self-infini-2}, the first inequality in \eqref{T}, and \eqref{algebraic-encore}.

Hence,  it follows from the comparison principle that
\begin{equation*}\label{comparisonbis}
u(t,x)\leq\psi(t+T,x)\leq w(t+T,x),\quad \forall (t,x)\in [0,+\infty)\times  \mathbb{R}.
\end{equation*}
For $t\geq t_\lambda$ and $x\in E_\lambda (t)$, it follows that
$w(t+T,x)\geq \lambda$ which, using the expression for $w$ transfers into
$$
\lambda (t+T)^{\frac{1}{\beta-1}}\leq \varphi\left(\frac{x}{(t+T)^{\frac{\beta-m}{2(\beta-1)}}}\right).
$$
From the properties of $\varphi$ we infer that $\frac{x}{(t+T)^{\frac{\beta-m}{2(\beta-1)}}}\to \overline{z_0}$ as $t\to+\infty$ and thus, from \eqref{self-z0},
$$
\lambda (t+T)^{\frac{1}{\beta-1}}\leq \frac{K_0}{\left(\frac{x}{(t+T)^{\frac{\beta-m}{2(\beta-1)}}}-\overline{z_0}\right)^{\frac{1}{\beta-1}}},
$$
for $t$ large enough. Hence
$$
x\leq (t+T)^{\frac{\beta-m}{2(\beta-1)}}\left(\overline{z_0}+ \left(\frac{K_0}{\lambda}\right)^{\beta-1}\frac{1}{t+T}\right)
<(1+\ep)\overline{z_0}t^{\frac{\beta-m}{2(\beta-1)}}=:x^{+}(t),
$$
 for $t\geq  T_{\lambda,\ep}$ chosen sufficiently large. This concludes the proof of  the upper bound in \eqref{improved}. \qed
\begin{rem}\label{rem:up2}
We point out that \eqref{def:self-2} provides a supersolution whatever the choice of a self-similar solution $\varphi$. In particular, all the asymptotics which we established in Section~\ref{s:self} are necessary only for the construction of the subsolution, which as usual is the more intricate part of the proof.

Therefore in the regime $m+\frac 2\alpha \leq \beta <2-m$, $\frac{1}{1-m}<\alpha<\frac{2}{1-m}$, replacing $\varphi$ by $\widetilde{\varphi}$ a solution of~\eqref{eq-self-similar} blowing up at some point $\widetilde{z}_0$ and decaying at infinity with asymptotics
$$
o(1) = \widetilde \varphi (z)\geq \frac{\overline{C}}{z^\alpha} \quad \mbox{ as } z \to +\infty,  
$$
one can repeat a similar argument to the above to find an upper estimate on the position of any level set in the form of \eqref{improved}; this was announced in Remark~\ref{rem:up1}. We also refer to Remark~\ref{rem:up3} below for the existence of such a function~$\widetilde{\varphi}$. 
\end{rem}

\section{Actual construction of the self-similar solution}\label{s:construction}

As far as the construction of self-similar solutions is concerned, let us mention the strategy of \cite{Guo-95} --- see also   \cite{Har-Wei-82}, \cite{Guo-Guo-01} for related results--- which mainly consists in using an integral formulation of the problem. Because of a non integrable singularity in the problem under investigation, this seems quite impracticable and we therefore adopt a different approach through sub and supersolutions.

Another difficulty is that we are looking for both limiting behaviours to be, in some sense, critical, and therefore we have to \lq\lq shoot" simultaneously in both directions. To do so, we start by looking for a solution that has the appropriate blow-up profile, with the blow-up point $z_0$ being any positive number. We will then show that if $z_0$ is too large, this blow-up solution decays \lq\lq slowly'' at infinity, while if $z_0$ is small it decays \lq\lq quickly''. This will lead us to find a particular $z_0$ where the solution has both wanted asymptotics at the blow-up point and at infinity.

\subsection{Comparison principles}

As mentioned above, we will use sub and supersolutions to construct the self-similar solution of Section \ref{s:self}. We state here some properties that we will use extensively. First, we say that $\psi$ is a subsolution of \eqref{eq-self-similar} if it satisfies the differential inequality
$$
-\frac{1}{\beta-1}\left(\psi(z)+\frac{\beta-m}{2}z\psi'(z)\right)\leq (\psi^m)''(z)+r\psi^\beta(z).
$$
Similarly, $\psi$ is a supersolution if the opposite inequality holds. Due to the singularity of the equation as $\psi \to 0$, we will only consider positive (sub and super) solutions.

Let us already point out that any \lq \lq shift to the left" of a decreasing subsolution remains a subsolution, whereas any \lq \lq shift to the right" of a decreasing supersolution remains a supersolution. This follows from a straightforward computation. For later use, we state this in the following proposition.

\begin{prop}[Shifting sub and supersolutions]\label{prop:subsuper_shift}
Let $\psi$ be a decreasing function. If $\psi$ is a subsolution (resp. a supersolution) of~\eqref{eq-self-similar}, then for any $a>0$, the  shifted function $\psi (\cdot +a)$ (resp. $\psi (\cdot -a)$) is also a subsolution (resp. a supersolution) of~\eqref{eq-self-similar}.
\end{prop}

Our main tool will be a comparison principle, which we establish through a sliding argument and thanks to the previous proposition.

\begin{prop}[A comparison principle]\label{prop:ode_comp} Let $\psi_1$ and $\psi_2$ be respectively a sub and a supersolution of \eqref{eq-self-similar} on an interval~$I$. Furthermore, we assume that both functions are decreasing.
\begin{enumerate}[$(i)$]
\item If $\psi_1 (z_1) = \psi_2 (z_1)$ and $\psi_1 ' (z_1) > \psi_2 ' (z_1)$ for some $z_1 \in I$, then $\psi_1 > \psi_2$ in $(z_1,+\infty) \cap I$.
\item If $\psi_1 (z_1) = \psi_2 (z_1)$ and $\psi_1 ' (z_1) < \psi_2 ' (z_1)$ for some $z_1 \in I$, then $\psi_1 > \psi_2$ in $(-\infty,z_1) \cap I$.
\end{enumerate}
\end{prop}

\begin{proof}
Let us prove the first statement. We proceed by contradiction and assume there is some $z > z_1$ where $\psi_1 (z) \leq \psi_2 (z)$. Due to the inequality $\psi_1 ' (z_1) > \psi_2 ' (z_1)$, clearly there exists $z_2 >z_1$ such that $\psi_1 (z_2) = \psi_2 (z_2)$ and $\psi_1 > \psi_2$ in $(z_1, z_2)$. Since both functions are decreasing, we can define
$$
a_0:=\inf \{a>0:  \psi_2 (\cdot -a) > \psi_1 \mbox{ in } (z_1+ a,z_2) \}>0.
$$
Hence, $\psi_2 (\cdot -a_0) - \psi_1$ reaches a zero minimum value at some point in $(z_1 + a_0, z_2)$. Now according to Proposition~\ref{prop:subsuper_shift}, for any $a>0$, the function $\psi_2 (\cdot -a)$ is also a supersolution of \eqref{eq-self-similar}. Thus, evaluating the differential inequalities satisfied by $\psi_2 (\cdot - a_0)$ and $\psi_1$ at this zero minimum value, one reaches a contradiction. We conclude that $\psi_1 > \psi_2$ in $(z_1, +\infty) \cap I$.

The second statement can be proved similarly and we omit the details.
\end{proof}

\begin{rem}\label{rem:comp1}
The monotonicity assumption on the functions $\psi_1$ and $\psi_2$ in Proposition~\ref{prop:ode_comp} can be weakened. As a matter of fact, in the sequel we will occasionally use slightly different sliding arguments to reach a similar comparison property.

Let us also notice that, by continuity of solutions of the ODE \eqref{eq-self-similar} with respect to boundary conditions, if in statement $(i)$ one only assumes that $\psi_1 ' (z_1) \geq \psi_2 ' (z_1)$, then one can still deduce that $\psi_1 \geq  \psi_2$ in $(z_1,+\infty) \cap I$. The same remark holds for statement~$(ii)$.
\end{rem}

This comparison principle can also be extended to the strict ordering of blow-up points. 

\begin{prop}[Ordering blow-up points]\label{prop:ode_comp_blowup}
Let $\psi_1$ and $\psi_2$ be respectively a sub and a supersolution of \eqref{eq-self-similar}. Furthermore, we assume that both functions are decreasing, and that $\psi_2$ blows up at some point $Z_2$.

If $\psi_1 (z_1) = \psi_2 (z_1)$ and $\psi_1 ' (z_1) \leq \psi_2 ' (z_1)$ for some $z_1$ in both their intervals of definition, then either $\psi_1 \equiv \psi_2$ on $(Z_2,z_1)$ or $\psi_1$ blows up at some point $Z_1 > Z_2$.
\end{prop}
\begin{proof}
By the previous proposition and the subsequent remark, we already know that $\psi_1 \geq \psi_2$ on the left of $z_1$ (and in the intersection of both their intervals of definition). In particular $\psi_1$ blows up at some $Z_1 \geq Z_2$ and it only remains to show that $\psi_1 \not \equiv \psi_2$ implies that $Z_1 > Z_2$. We proceed by contradiction and assume that there exists some point $\tilde{z}$ on the left of $z_1$ where $\psi_1 (\tilde{z}) >  \psi_2 (\tilde{z})$, yet $Z_1 = Z_2$. For any $a>0$, the shifted to the left $\psi _{1}(\cdot +a)$ subsolution is (strictly) smaller than $\psi_2$ in neighborhoods of $Z_2$ (thanks to $Z_1=Z_2$) and of $z_1-a$ (thanks to the monotonicity assumption). Moreover, if $a$ is small enough, then $\psi_1 (\tilde{z} + a) - \psi_2 (\tilde{z})$ has to be positive, so that both functions intersect at least twice. Using again Proposition~\ref{prop:ode_comp} and continuity of solutions of the ODE w.r.t. boundary conditions, as combined in Remark \ref{rem:comp1}, we see that $\psi_1 (\cdot +a) \geq \psi_2$ on a neighborhood of $Z_2$, contradicting the fact that it blows up at $Z_2 -a$.
\end{proof}
We point out that an immediate corollary of Proposition~\ref{prop:ode_comp_blowup} is that two different solutions blowing up at the same point cannot intersect.

\subsection{Existence of a solution with the appropriate blow-up profile}\label{ss:bup}

In this subsection, we fix $z_0$ any positive real number. Our construction relies on several sub and supersolutions which we detail below. For $C>0$, we let 
$$
\psi_{0,C}(z): = \frac{C}{(z-z_0)^{\frac{1}{\beta - 1}  }},
$$
which we aim at plugging in equation \eqref{eq-self-similar}. We compute
$$-\frac{1}{\beta-1}\left(\psi_{0,C} (z)+\frac{\beta-m}{2}z\psi_{0,C} '(z)\right)= -\frac{C}{\beta-1}\left(\frac{1}{(z-z_0)^{\frac{1}{\beta -1}}} - \frac{\beta-m}{2 (\beta -1)} \frac{z}{z-z_0} \frac{1}{(z-z_0)^{\frac{1}{\beta -1}}} \right)$$
while 
$$(\psi_{0,C}^m)''(z)+r\psi_{0,C}^\beta(z) =  \frac{m}{\beta-1} \frac{m+\beta-1}{\beta-1} \frac{C^m}{(z-z_0)^{\frac{m +2 \beta -2}{\beta -1}}} + r \frac{C^\beta}{(z-z_0)^{\frac{\beta}{\beta-1}}}.$$
Thanks to the fact that $m + 2 \beta - 2 <\beta$, we find that there is some 
\begin{equation}\label{def:C0}
C_0 := \left(\frac{z_0 (\beta-m)}{2 r (\beta -1)^2}\right)^{\frac{1}{\beta-1}}>0
\end{equation}
such that, for any $C > C_0$, $\psi_{0,C}$ satisfies
$$-\frac{1}{\beta-1}\left(\psi_{0,C} (z)+\frac{\beta-m}{2}z\psi_{0,C} '(z)\right) < (\psi_{0,C}^m)''(z)+r\psi_{0,C}^\beta(z) ,$$
in a neighborhood of $z_0$ (strict subsolution) while, for any $C< C_0$, $\psi_{0,C}$ satisfies the opposite inequality (strict supersolution).

Let us be more specific concerning the neighborhoods of $z_0$. For any $\gamma>1$, let $k_0:=\frac{C_0}{\gamma}$ and $K_0:= \gamma C_0$. Then $\psi_{0,K_0}$ is a subsolution of the ODE \eqref{eq-self-similar} if
$$ \frac{(\gamma^\beta - \gamma) r C_0^\beta}{(z-z_0)^\frac{\beta}{\beta-1}} +\frac{m}{\beta-1} \frac{m+\beta-1}{\beta-1} \frac{(\gamma C_0)^m}{(z-z_0)^{\frac{m +2 \beta -2}{\beta -1}}}> - \frac{\gamma C_0}{\beta-1} \frac{1}{(z-z_0)^{\frac{1}{\beta -1}}} \left( 1 - \frac{\beta-m}{2 (\beta -1)}  \right) .$$
Since the second term in the left hand side (diffusion term) and the right hand side are positive, a sufficient condition is given by
$$
(\psi_{0,K_0} (z))^{\beta-1}>  \frac{\gamma^\beta }{\gamma^\beta - \gamma} C_1 (r,\beta,m),
$$
where $C_1$ is positive. Notice that the right hand side does not depend on $z_0$. Notice also that, recalling the expression of $C_0$ above, we can rewrite this as
\begin{equation}
\label{kappa}
0 < z - z_0  \leq   \frac{\gamma^\beta - \gamma}{\gamma C_2 (r ,\beta,m)} z_0=:\kappa z_0.
\end{equation}
One can also make more precise the neighborhood where $\psi_{0,k_0}$ is a supersolution. Here the diffusion term makes things a bit more difficult. However, we can still find some $\delta (z_0) \in (0,\kappa z_0)$, depending continuously on $z_0$, such that $\psi_{0,k_0}$ is a supersolution on $(z_0,z_0 + \delta (z_0)]$.

We will need a third function, which we define as
$$\tilde{\psi}(z) := \frac{1}{(z-z_0)^{\gamma_1}}, \qquad 0 < \gamma_1 < \frac{1}{1-m}.$$
Up to decreasing $\delta (z_0)$ and without loss of generality, the function $\tilde{\psi}$ satisfies
$$
-\frac{1}{\beta-1}\left(\tilde{\psi} (z)+\frac{\beta-m}{2}z \tilde{\psi} '(z)\right) < (\tilde{\psi}^m)''(z)+r\tilde{\psi}^\beta(z),
$$
in the interval $(z_0, z_0 +\delta(z_0)]$ (strict subsolution), as well as 
$$
\tilde{\psi} < \psi_{0,k_0} < \psi_{0,K_0}.
$$
We recall that, according to Proposition~\ref{prop:subsuper_shift}, any shift of $\tilde{\psi}$ or $\psi_{0,K_0}$ to the left remains a subsolution, while any shift to the right of $\psi_{0,k_0}$ remains a supersolution.\medskip

Equipped with these sub and supersolutions, we can now proceed. We first construct a solution which blows up but not necessarily with the appropriate asymptotics: choose any
$$\theta \in \left( \tilde{\psi} (z_0 + \delta (z_0)), \psi_{0,k_0} (z_0 + \delta (z_0)) \right),$$
and for any $\xi \in \mathbb{R}$, define $\varphi_{\theta,\xi}$ as the solution of the ODE \eqref{eq-self-similar} with initial conditions
$$
\varphi (z_0 + \delta (z_0)) = \theta, \qquad \varphi ' (z_0 + \delta (z_0)) = - \xi.
$$

\begin{cla}\label{claim1}
For any $\theta \in \left( \tilde{\psi} (z_0 + \delta (z_0)), \psi_{0,k_0} (z_0 + \delta (z_0)) \right)$, there exists $\xi \in \mathbb{R}$ such that
$$\tilde{\psi} < \varphi_{\theta,\xi} < \psi_{0,k_0}\; \text{ in } (z_0,z_0 + \delta (z_0)].
$$
In particular, $\varphi_{\theta,\xi}$ blows up at $z_0$.
\end{cla}

\begin{proof} Note that, by a strong maximum type argument, it is enough to find $\xi$ such that
$\tilde{\psi} \leq \varphi_{\theta,\xi} \leq \psi_{0,k_0}$.  We define 
$$
\Xi_1 := \{ \xi \in \R  :  \varphi_{\theta, \xi} < \tilde{\psi} \mbox{ at some point in $(z_0,z_0 +\delta (z_0))$}\},
$$
$$
\Xi_2 := \{ \xi \in \R :  \varphi_{\theta, \xi} > \psi_{0,k_0} \mbox{ at some point in $(z_0,z_0 +\delta (z_0))$}\}.
$$

We first prove that neither set is empty.  Indeed, in view of the ODE \eqref{eq-self-similar} a solution cannot reach a positive minimum and thus any $\xi <0$ enforces the solution $\varphi_{\theta,\xi}$ to cross $\tilde \psi$, so that $(-\infty,0)\subset \Xi _1$. On the other hand a direct computation shows that, if the slope $p>0$ is large enough, then the linear function $w(z):=-p(z-(z_0+\delta (z_0)))+\theta$ is a supersolution on an interval $(z_p,z_0+\delta (z_0))$ where it crosses $\psi _{0,k_0}$ exactly once. Now choose $\xi >p$, and let us prove that $\varphi_{\theta,\xi}$ also crosses $\psi_{0,k_0}$ in the same interval. To do so, we use a sliding argument: notice that the solution $\varphi_{\theta,\xi}$ may not be decreasing, so that Proposition~\ref{prop:ode_comp} does not apply directly. Nonetheless, we proceed by contradiction and assume that $\varphi_{\theta,\xi} \leq \psi_{0,k_0}$ in $(z_p,z_0 + \delta (z_0))$. In particular, the supremum of $\varphi_{\theta,\xi}$ in the same interval is less than $w (z_p)$. One can then reproduce the exact same argument as in the proof of Proposition~\ref{prop:ode_comp} to obtain a critical shift $a_0$ such that $w (\cdot - a_0) - \varphi_{\theta,\xi}$ admits a zero minimum in the interval $(z_p + a_0, z_0 +\delta (z_0))$. This contradicts the differential inequality and equality satisfied by both functions $w$ and $\varphi_{\theta, \xi}$. It follows that $\Xi_2$ is not empty either.

By continuity of the solutions of the ODE w.r.t. the slope parameter $\xi$, $\Xi_1$ and $\Xi _2$ are open sets.

 Also, by a comparison argument, if $\xi \in \Xi_1$ then $(- \infty,\xi] \subset \Xi_1$. Indeed, choose any $\xi '< \xi$ and assume by contradiction that $\varphi_{\theta,\xi '} \geq \tilde{\psi}$ on the left of $z_0$. In particular, $\varphi_{\theta,\xi'}$ blows up at some point $z_{\xi '} \geq z_0$ and, since as we explained above it cannot reach a positive minimum, it has to be decreasing on its interval of definition. We again use a sliding argument and find some $a_0 >0$ so that the function $\varphi_{\theta,\xi'} (\cdot - a_0) - \varphi_{\theta,\xi}$ reaches a zero minimum inside the interval $(z_{\xi'}+a_0,z_0 + \delta (z_0))$. This gives a contradiction and we conclude that $\xi' \in \Xi_1$.

By another comparison argument, we also show that $\xi \in \Xi_2$ implies $[\xi,+\infty) \subset \Xi_2$. Indeed, choose $\xi ' > \xi$ and assume by contradiction that $\varphi_{\theta,\xi '} \leq \psi_{0,k_0}$. Using again the fact that solutions may not admit a positive minimum, we conclude that $\varphi_{\theta,\xi}$ is decreasing on some interval $(z_{\xi}, z_0 + \delta (z_0))$ where it crosses exactly once $\psi_{0,k_0}$. In particular, $\varphi_{\theta,\xi} (z_{\xi})$ is larger than the supremum of $\varphi_{\theta,\xi'}$ over the interval $(z_{\xi}, z_0 +\delta (z_0))$. Proceeding as above, one can extend the proof of Proposition~\ref{prop:ode_comp}  and reach a contradiction.

Finally, we conclude that there exist $\xi_1 \leq \xi_2$ such that 
$\Xi_1 = (-\infty,\xi_1)$, $\Xi_2 = (\xi_2,+\infty)$.

Now we use a sliding argument to prove that  $\Xi_1 \cap \Xi_2 = \emptyset$: if this is not true then we are equipped with  a solution $\varphi_{\theta,\xi}$ which crosses twice either $\psi_{0,k_0}$ or $\tilde{\psi}$. Consider the former case, there is $z_1 < z_2$ such that $\varphi_{\theta,\xi}-\psi_{0,k_0}> 0$ on $(z_1,z_2)$ and $\varphi_{\theta,\xi}-\psi_{0,k_0}= 0$ on $\{z_1,z_2\}$. Then we can define
$$
a_0:=\inf \{a>0: \varphi_{\theta,\xi}\leq \psi_{0,k_0}(\cdot-a) \mbox{ in } (\max\{ z_1, z_0 + a\},z_2) \}>0.
$$
Hence, thanks to the monotonicity of $\psi_{0,k_0}$ and the fact that $\psi_{0,k_0}(z)\to +\infty$ as $z \to z_0$, we conclude that $\psi_{0,k_0}(\cdot-a_0)-\varphi_{\theta,\xi}$ reaches a zero minimum value at some point, which is contradicted by the differential inequality for the supersolution $\psi_{0,k_0}(\cdot- a_0)$ and the ODE for the solution $\varphi_{\theta,\xi}$. Therefore, it remains to rule out the case when $\varphi_{\theta,\xi}$ only crosses $\psi_{0,k_0}$ once, and $\tilde{\psi}$ twice. Recalling that $\varphi_{\theta,\xi}$ cannot change monotonicity more than once, it then follows that it is a decreasing function. A straightforward use of Proposition~\ref{prop:ode_comp} (which is now applicable) contradicts the fact that $\psi_{0,k_0}$ crosses $\tilde{\psi}$ twice. Finally $\xi_1 \leq \xi_2$ and the claim follows by taking any $\xi \in [\xi_1 , \xi_2]$.
\end{proof}

Next we introduce 
$$\theta_* = \sup \{ \theta: \exists \xi \mbox{ s.t. the solution } \varphi_{\theta,\xi} < \psi_{0,K_0} \mbox{ in } (z_0,z_0 + \delta (z_0)] \mbox{ and blows up at $z_0$}\}.$$
The set of such $\theta$ is not empty by the above claim (recall that $\psi_{0,k_0} < \psi_{0,K_0}$), and obviously it is bounded from above as the inequality fails when $\theta \geq \psi_{0,K_0} (z_0 + \delta (z_0))$. It follows that $\theta_*$ is well-defined and finite. Now take sequences
$$\theta_n \nearrow \theta_* \mbox{ and } \xi_n,$$
where $\xi_n$ is chosen so that $\varphi_{\theta_n,\xi_n}$ blows up at $z_0$ and lies below $\psi_{0,K_0}$. 
\begin{cla}
The sequence $\xi_n$ is bounded. Hence, up to extraction of a subsequence, $\xi_n \to \xi_*$.
\end{cla}

\begin{proof} If $\xi _n <0$ the solution $\varphi_{\theta _n,\xi _n}$, which blows up at $z_0$, has to reach a positive minimum at some point, where we test the equation to reach a contradiction. Hence $\xi _n\geq 0$ and moreover $\varphi_{\theta_n,\xi_n}$ is a decreasing function. Using yet another comparison argument, we obtain that for any $\xi \leq \xi_n$ the solution $\varphi_{\theta_n, \xi}$ lies below $\psi_{0,K_0}$: if not and using again the fact that solutions cannot reach a positive minimum at any point, then $\varphi_{\theta_n,\xi}$ has to be decreasing in some left interval of $z_0 + \delta (z_0)$ where it crosses $\psi_{0,K_0}$, and applying Proposition~\ref{prop:ode_comp} to the pair of functions $\varphi_{\theta_n,\xi_n}$ and $\varphi_{\theta_n,\xi}$ one immediately reaches a contradiction.

Now proceed by contradiction and assume that $\xi_n \to +\infty$ (even up to extraction of a subsequence). Then by a limiting argument we get that for any $\xi \in \mathbb{R}$, the solution $\varphi_{\theta_*,\xi}$ lies below $\psi_{0,K_0}$, which is impossible (see the non emptiness of $\Xi_2$ in the proof of Claim~\ref{claim1}). The claim is proved.
\end{proof}

We let $ n \to +\infty$, and by continuity of the solutions of the ODE, we have 
$$
\varphi_{\theta_*, \xi_*} \leq \psi_{0,K_0}\quad \text{ in } (z_0,z_0+\delta (z_0)].
$$
Let us check that $\varphi_{\theta_*,\xi_*}$ blows up at $z_0$. Clearly the above inequality implies that it cannot blow up on the right of $z_0$. Now proceed by contradiction and assume that $\varphi_{\theta_*,\xi_*}$ is finite at $z_0$. Using again the continuity of solutions of the ODE, we have a small open neighborhood of $(\theta_*,\xi_*)$ such that the solution is again finite at $z_0$: this is a clear contradiction with our construction. We conclude, as announced, that $\varphi_{\theta_*,\xi_*}$ blows up at $z_0$.

As we have pointed out several times, if $z_*$ is a critical point of any (positive) solution, the equation \eqref{eq-self-similar} yields that $z_*$ is a strict local maximum point. In particular, blow-up cannot occur on the left of~$z_*$, and hence the following holds.

\begin{cla}
The function $\varphi_{\theta_*,\xi_*}$ is decreasing on its interval of definition.
\end{cla}

Before we proceed, let us extend the previous upper inequality: we show that 
\begin{equation}\label{enlarged}
\varphi_{\theta_*, \xi_*} < \psi_{0,K_0}\quad \text{ in } (z_0, z_0 + \kappa  z_0),
\end{equation}
where $\kappa$ comes from \eqref{kappa} and is such that $\psi_{0,K_0}$ is still a subsolution in this enlarged interval $(z_0,z_0+\kappa z_0)$. Assume by contradiction that there is a contact point $z_1 \in (z_0, z_0 + \kappa z_0)$ and without loss of generality that $\varphi_{\theta_*,\xi_*} \leq \psi_{0,K_0}$ in $(z_0,z_1)$. Since we are equipped with a strict subsolution and a solution we also have that $\varphi_{\theta_*,\xi_*}\not \equiv \psi_{0,K_0}$ on $(z_0,z_1)$. Proposition~\ref{prop:ode_comp_blowup} immediately contradicts the fact that both functions blow up at the same point $z_0$. Therefore \eqref{enlarged} holds.
\medskip

We are now ready to prove that $\varphi_{\theta_* , \xi_*}$ blows up with the appropriate behaviour, in the sense that
\begin{equation}\label{step}
\varphi_{\theta_* ,\xi_*} \geq \psi_{0,k_0}\quad \text{ in } (z_0,z_0 +\delta (z_0)).
\end{equation}

We proceed again by contradiction and assume that $\varphi_{\theta_*,\xi_*} (z_0 + \delta ') < \psi_{0,k_0} (z_0 + \delta')$ for some $\delta ' \in (0,\delta (z_0))$. The idea is to show that $\varphi_{\theta_*,\xi_*}$ falls into the value range (between $\tilde{\psi}$ and $\psi_{0,k_0}$) where blow-up is expected with \lq\lq slow'' asymptotics, as the solutions we have constructed in Claim \ref{claim1}. Therefore, small perturbations should also be in the same value range, leading to a contradiction with the \lq\lq critical'' choice of $\theta_*$.

First, due to Proposition~\ref{prop:ode_comp_blowup}, the function $\varphi_{\theta_*,\xi_*}$ cannot intersect $\psi_{0,k_0}$ in $(z_0,z_0 +\delta ')$ (it would contradict the fact that blow-up occurs at $z_0$). It follows that 
$$
\varphi_{\theta_*, \xi_*} < \psi_{0,k_0}\quad \text{ in } (z_0,z_0 + \delta ').
$$
Note that we can write a strict inequality thanks to the strong maximum principle. In a similar fashion, as we know that $\varphi_{\theta_*, \xi_*} (z_0 + \delta) >  \tilde{\psi} (z_0 + \delta (z_0))$, and as it cannot intersect $\tilde{\psi}$, we have
$$\varphi_{\theta_*,\xi_*} > \tilde{\psi}\quad \text{ in } (z_0,z_0 + \delta (z_0)).
$$
We again used Proposition~\ref{prop:ode_comp_blowup}, and a strong maximum principle for this inequality to be strict.

Now consider $\hat{\varphi}$ the solution of the ODE \eqref{eq-self-similar} with initial conditions
$$
\hat{\varphi} (z_0 + \delta' ) = \hat{\theta}, \qquad  \hat{\varphi} '(z_0 + \delta ' ) = - \hat{\xi}.
$$
Here we choose $\hat{\theta} > \varphi_{\theta_*,\xi_*}(z_0+\delta')$ but very close, and by a similar argument as that of Claim \ref{claim1}, we can find a $\hat{\xi}$ such that $\hat{\varphi}$ remains between $\tilde{\psi}$ and $\psi_{0,k_0}$ in the interval $(z_0 , z_0 + \delta ')$. Since both solutions blow up at the same point and by Proposition~\ref{prop:ode_comp_blowup}, we have that $\hat{\varphi}$ and $\varphi_{\theta_*, \xi_*}$ cannot intersect. In particular, noting that $\hat{\varphi}$ is also defined on the right of $z_0 + \delta '$, we get that 
$$
\theta_* = \varphi_{\theta_* , \xi_*} (z_0 + \delta (z_0)) < \hat{\varphi} (z_0 + \delta (z_0)).
$$
In a similar fashion, it follows from Proposition~\ref{prop:ode_comp_blowup} and the fact that $\hat \varphi \leq \psi_{0,k_0} < \psi_{0,K_0}$ in a neighborhood of $z_0$, that $\hat \varphi < \psi_{0,K_0}$ in the whole interval $(z_0,z_0 + \delta (z_0)]$. Therefore the existence of the function $\hat \varphi$ contradicts our choice of $\theta_*$, and this contradiction concludes the proof of \eqref{step}.

\medskip

We sum up our result in the next proposition.
\begin{prop}[Blow-up solutions]\label{prop:blow-up}
For any $z_0 >0$, any $k_0 < C_0 < K_0$, there exists a decreasing solution $\varphi_{z_0}$ of the ODE \eqref{eq-self-similar} such that
\begin{equation}\label{validity}
\frac{k_0}{(z-z_0)^{\frac{1}{\beta-1}}} \leq \varphi_{z_0} (z) \leq \frac{K_0}{(z-z_0)^{\frac{1}{\beta-1}}}
\end{equation}
in some right neighborhood of $z_0$.

Furthermore, when $\gamma k_0 = \frac{K_0}{\gamma} = C_0$ with $\gamma >1$, the upper inequality holds true on an explicit interval, see \eqref{kappa} and \eqref{enlarged}. The lower inequality holds in an interval $(z_0,z_0 + \delta (z_0))$, see \eqref{step}, where $\delta (z_0)$ can be chosen to be continuous with respect to $z_0$.
\end{prop}  

\begin{rem}
In particular, the above statement implies that, given a sequence $z_n \to z_0$, the associated solutions $\varphi_{z_n}$ satisfy estimates on intervals that \lq\lq do not disappear'' when passing to the limit. In particular, any limit of the solutions $\varphi_{z_n}$ (which exists by usual estimates, up to extraction of a subsequence) has the wanted blow-up behaviour too.
\end{rem}

Let us now turn to the estimates on the  derivatives of $\varphi_{z_0}$, that is \eqref{to-be-proved-2abis} and \eqref{to-be-proved-2b}. Take any $z_1 \in (z_0,z_0 + \delta (z_0))$ close enough to $z_0$ such that $ \psi_{0,k_0} (z_1)<\varphi_{z_0} (z_1) < \psi_{0,K_0} (z_1)$ and $\varphi_{z_0} (z_1) > \psi_{0,K_0} (z_0 + \delta (z_0))$. Note that these inequalities can be made strict, up to an arbitrarily small change in the constants $k_0$, $K_0$, which has no incidence here. Then there exists some shift $b\in (0,z_0 + \delta (z_0) - z_1)$ such that 
\begin{equation}
\label{shift-b}
\psi_{0,K_0} (z_1 + b) = \varphi_{z_0} (z_1),
\end{equation}
and moreover $\psi_{0,K_0} (\cdot + b)$ satisfies the differential inequality
$$-\frac{1}{\beta-1}\left(\psi_{0,K_0} (z+b)+\frac{\beta-m}{2}(z+ b) \psi_{0,K_0} '(z+b)\right) < (\psi_{0,K_0}^m)''(z+b)+r\psi_{0,K_0}^\beta(z+b) , $$
and thus
$$-\frac{1}{\beta-1}\left(\psi_{0,K_0} (z+b)+\frac{\beta-m}{2}z  \psi_{0,K_0} '(z+b)\right) < (\psi_{0,K_0}^m)''(z+b)+r\psi_{0,K_0}^\beta(z+b) , $$
at least in $(z_0,z_0 + \delta (z_0))$. From this we have
\begin{equation}
\label{claim-pentes}
\varphi_{z_0} ' (z_1) < \psi_{0,K_0} ' (z_1 + b).
\end{equation}
Indeed, if \eqref{claim-pentes} is not true, then by Proposition~\ref{prop:ode_comp} and the continuity of solutions with respect to boundary conditions (see Remark~\ref{rem:comp1}), we conclude that $\varphi_{z_0} \leq \psi_{0,K_0} (\cdot + b)$ on the left of $z_1$, which contradicts the fact that it blows up at $z_0$.
Now from $\psi_{0,K_0} ' =  -\frac{K_0^{1-\beta}}{\beta-1} \psi_{0,K_0}^\beta$ and \eqref{shift-b}, inequality~\eqref{claim-pentes} is recast as
$$
\varphi_{z_0} ' (z_1) < - \frac{K_0^{1-\beta}}{\beta-1}( \varphi_{z_0} (z_1))^\beta,
$$
and is true for any $z_1$ close enough to $z_0$. Repeating the same argument with the supersolution $\psi_{0,k_0}$, we reach the conclusion that
\begin{equation}\label{reach}
- \frac{k_0^{1-\beta}}{\beta-1}( \varphi_{z_0} (z_1))^\beta < \varphi_{z_0} ' (z_1) < - \frac{K_0^{1-\beta}}{\beta-1}( \varphi_{z_0} (z_1))^\beta,
\end{equation}
which concludes the proof of \eqref{to-be-proved-2abis}.

We now prove \eqref{to-be-proved-2b}. Let $\delta>0$ be given. From \eqref{eq-self-similar} and straightforward computations we deduce that
$$
m\varphi ''(z) \varphi ^{m-\beta-1}(z)=-\frac{1}{\beta-1}\frac{1}{\varphi^{\beta-1}(z)}+m(1-m)\varphi^{m-2-\beta}(z)(\varphi ')^2(z)-\frac{\beta-m}{2(\beta-1)}z\varphi'(z)\varphi^{-\beta}(z)-r.
$$
From \eqref{reach} and the expression of $C_0$ in \eqref{def:C0}, we see that, if $\max(K_0-C_0,C_0-k_0)$ is small enough, then $\vert -\frac{\beta-m}{2(\beta-1)}z \varphi'(z)\varphi^{-\beta}(z)-r\vert \leq \frac{\delta m}{2}$ in a right neighborhood of $z_0$, say $(z_0,z_1)$. Next, from~\eqref{reach} again and the blow-up of $\varphi_{z_0}$ at $z_0$, we see that the first two terms in the above right hand side can be made small for any $z\in(z_0,z_1)$, up to reducing $z_1$. This concludes the proof of \eqref{to-be-proved-2b}.

\subsection{Behaviour of $\varphi_{z_0}$ at infinity}\label{ss:beh-infinity}

We will next investigate the behaviour of $\varphi_{z_0}$ as $z \to +\infty$, depending on the choice of the blow-up point $z_0$. Let us first introduce some sub and supersolutions decaying to 0.

For any $C>0$, we define
$$
\psi_{\infty,C} (z) := \frac{C}{z^{\frac{2}{1-m}}}.
$$
A straightforward computation shows that there is $C_\infty=C_\infty(m)>0$ so that $\psi_{\infty,C_\infty}$ satisfies the ODE \eqref{eq-self-similar} without the reaction term (which is negligible at infinity), that is \eqref{eq-self-similar} with $r=0$. In particular, $\psi_{\infty,C_\infty}$ satisfies the differential inequality
$$
-\frac{1}{\beta-1} \left(\psi_{\infty,C_\infty} (z) + \frac{\beta-m}{2} z \psi_{\infty,C_\infty} ' (z) \right) <  (\psi_{\infty,C_\infty}^m)'' (z) + r \psi_{\infty,C_\infty}^\beta (z),
$$
thus being a strict subsolution to \eqref{eq-self-similar}. As before (see Proposition~\ref{prop:subsuper_shift}), any shift to the left, i.e. any $\psi_{\infty,C_\infty} (\cdot + Z)$ with $Z >0$, satisfies the same differential inequality. We choose $Z$ such that the shifted function intersects the vertical axis with a value larger than $\gamma \left(\frac{\beta-m}{2 \kappa r(\beta-1)^2}\right)^{\frac{1}{\beta-1}}$, i.e.
\begin{equation}\label{gamma1}
\psi_{\infty,C_\infty} (Z) > \gamma \left(\frac{\beta-m}{2 \kappa r(\beta-1)^2}\right)^{\frac{1}{\beta-1}},
\end{equation}
where $\gamma >1$ is the scaling parameter defined in subsection \ref{ss:bup}, namely $k_0=\frac{C_0}{\gamma}$, $K_0=\gamma C_0$. We denote by $\psi_\infty=\psi_{\infty,C_\infty}(\cdot+Z)$ the obtained strict subsolution.

On the other hand, we also get that $\psi_{\infty,C}$ is a supersolution when $C > C_\infty$, and a subsolution when $C < C_\infty$, at least when $\psi_{\infty,C}(z)$ is small enough or equivalently when $z$ is large. Again we recall that, from {Proposition~\ref{prop:subsuper_shift}, any shift of a supersolution to the right remains a supersolution, while any shift to the left of a subsolution remains a subsolution. We will continue to use this extensively.

We first look, in the next two propositions, at the case when the blow-up point $z_0$ is large.

\begin{prop}[Far away slopes for large blow-up points]\label{prop:far-away}
Let $\eta >0$ be given. For any $z_0 >0$, let us define the point $z_1>z_0$ where 
$$
\varphi_{z_0} (z_1) = \eta.
$$
Then $\varphi_{z_0} ' (z_1)$ remains bounded as $z_0 \to +\infty$.
\end{prop}

\begin{proof}
Let us recall that $\varphi_{z_0}$ is decreasing. Hence $\varphi_{z_0}$ has to decay to some $l\geq 0$. Assume by contradiction that $l>0$, so that the solution exists until $+\infty$. Integrating~\eqref{eq-self-similar} from $z_*:=z_0+\delta (z_0)$ to $z$ we get, abbreviating $\varphi _{z_0}$ by $\varphi$,
\begin{eqnarray*}
(\varphi ^m)'(z)&=&(\varphi^m)'(z_*)-r\int _{z_*}^z\varphi^\beta(s)ds\\
&&-\frac{1}{\beta-1}\left(\frac{2-\beta+m}{2}\int_{z_*}^z\varphi(s)ds+\frac{\beta-m}{2}\left(z\varphi(z)-z_*\varphi(z_*)\right)\right).
\end{eqnarray*}
Since $2-\beta+m>0$ and $\beta-m>0$, letting $z\to +\infty$ we see that $\varphi'_{z_0}(z)\to -\infty$ as $z\to +\infty$, which is a contradiction. Hence $l=0$ and, in particular $z_1$ is well defined (uniquely) whatever the value of~$\eta$. It remains to find a lower bound on the derivative.

In order to obtain uniform bounds with respect to $z_0$, we first fix a reference point $\hat{z}_0 >0$, and define $\varphi_{\hat{z}_0}$ the corresponding solution constructed in the previous subsection, see Proposition \ref{prop:blow-up}. Then we know from \eqref{reach} that, on a neighborhood of $\hat{z}_0$ or equivalently, when $\varphi_{\hat{z}_0}$ is large enough, we have
$$
- \frac{{(\hat C_0/\gamma})^{1-\beta}}{\beta-1}( \varphi_{\hat{z}_0} (z))^\beta < \varphi_{\hat{z}_0} ' (z) < - \frac{(\gamma \hat C_0)^{1-\beta}}{\beta-1}( \varphi_{\hat{z}_0} (z))^\beta,
$$
where $\hat{C}_0=\hat C_0(\hat z_0)$ is the constant defined above, see \eqref{def:C0}, with $\hat{z}_0$ instead of $z_0$ (for which the constant is denoted $C_0$).

In the sequel, we work with $z_0>\hat z_0$ large enough so that
$$
\frac{C_0}{\gamma} > \gamma \hat{C}_0,
$$
which is possible since  $C_0 \to +\infty$ as $z_0 \to +\infty$, see \eqref{def:C0}.

From \eqref{reach}, we know that
$$
\varphi_{z_0} ' (z) >  - \frac{(C_0/\gamma)^{1-\beta}}{\beta-1}( \varphi_{z_0} (z))^\beta
$$
when $\varphi_{z_0} (z)$ is large enough. Now take $\eta$ large enough so that the above derivative estimate holds: this immediately gives a bound which does not depend on $z_0$ itself (recall that $1-\beta <0$).

It remains to consider the case when $\eta$ is not so large. This is where introducing $\varphi_{\hat{z}_0}$ turns out to be useful. Indeed, the three above inequalities imply that, for $\eta$ large enough, 
$$
\varphi_{\hat z_0}'(\hat z)< \varphi_{z_0}'(z)\leq 0,
$$
where $z$ and $\hat z$ are the points where respectively $\varphi _{z_0}$ and $\varphi_{\hat z_0}$ take the value $\eta$. In other words, $\varphi_{z_0}$ is less steep than $\varphi_{\hat{z}_0}$ at large enough values $\eta$. Let us show that this remains true at any value $\eta$. We proceed by contradiction and thus consider the largest $\eta_0>0$ such that
$$\varphi_{z_0} ' (z_1) = \varphi_{\hat{z}_0}' (\hat{z}_1),$$
where $z_1$ and $\hat{z}_1$ are the points where respectively $\varphi_{z_0}$ and $\varphi_{\hat{z}_0}$ take the value $\eta_0$. From the above,  $z_1 > \hat{z}_1$ must hold. In particular, by Proposition~\ref{prop:subsuper_shift} the shifted function $\varphi_{\hat{z}_0} (\cdot -  (z_1 - \hat{z}_1))$ is a supersolution of the ODE. Then, due to the degenerate zero of $\varphi_{z_0} - \varphi_{\hat{z}_0} (\cdot -  (z_1 - \hat{z}_1))$ at $z_1$, we have that it is nonnegative on a left neighborhood of $z_1$. But,
from the choice of $z_1$ above and since $\varphi_{z_0}$ is less steep than $\varphi_{\hat z_0}$ at values $\eta>\eta _0$, one must have $\varphi_{z_0} - \varphi_{\hat{z}_0} (\cdot -  (z_1 - \hat{z}_1))<0$ on the left of $z_1$. We have thus reached a contradiction. As announced, when $z_0$ is large, then $\varphi_{z_0}$ is  less steep than the fixed reference function $\varphi_{\hat z_0}$ at any value $\eta$, thus completing the proof.
\end{proof}

\begin{prop}[The decay from below for large blow-up points]\label{prop:slow-decay}
Let $K>C_\infty$ be given, where $C_\infty>0$ was defined in the beginning of subsection \ref{ss:beh-infinity}. Then, when $z_0>0$ is large enough, the solution $\varphi_{z_0}$ of the ODE \eqref{eq-self-similar} constructed in Proposition \ref{prop:blow-up} satisfies
\begin{equation}\label{eq:up}
\varphi_{z_0} (z) \geq \frac{K}{z^{\frac{2}{1-m}}},
\end{equation}
for any large $z$.
\end{prop}

\begin{proof} 
We recall that, since $K>C_\infty$,  the function $\psi_{\infty,K}(z)=\frac{K}{z^{\frac{2}{1-m}}}$ is a supersolution of the ODE \eqref{eq-self-similar} when it is small enough, say less than some value $\eta_K$. Our goal is to show that $\varphi_{z_0}$ is ``less steep'' than $\psi_{\infty,K}$ around some level set, which will lead to the wanted conclusion by Proposition~\ref{prop:ode_comp}.

Consider the points $z_1$ and $z_2$ such that $\varphi_{z_0} (z_1)  = \eta_K$ and $\varphi_{z_0} (z_2) = \frac{\eta_K}{2}$. As pointed out earlier, such points necessarily exist. Up to increasing~$z_0$, obviously~$z_1$ and $z_2$ can be made arbitrarily large so that $\psi_{\infty,K} <\varphi_{z_0}$ in $(z_1,z_2)$.

Now consider $z \in [z_1,z_2]$, and $\eta \in [\frac{\eta_K}{2},\eta_K]$ such that $\varphi_{z_0} (z ) = \eta$. We can find $S>0$ (depending on $z$) such that $\psi_{\infty,K} (z - S) = \eta$. If there is $z$ such that
$$
\varphi_{z_0} ' (z) > \psi_{\infty,K} '(z - S),
$$
then by Proposition~\ref{prop:ode_comp} we get the wanted inequality.

Let us thus assume the opposite inequality for all $z \in [z_1,z_2]$. In particular, we have that
$$
\max_{z\in[z_1,z_2]} \varphi_{z_0} '(z) \leq \max \left\{ \psi_{\infty,K} ' (y) :  \frac{\eta_K}{2} \leq \psi_{\infty,K} (y) \leq \eta_K \right\} <  0.
$$
Notice that this above bound does not depend on $z_0$, and thus it follows that the length $z_2 - z_1$ does not go to $+\infty$ as $z_0 \to +\infty$. Now integrating~\eqref{eq-self-similar} between $z_1$ and $z_2$, we find that
$$(\varphi_{z_0}^m)' (z_2) - (\varphi_{z_0}^m)' (z_1) = -r \int_{z_1}^{z_2} \varphi_{z_0}^\beta (s) ds - \frac{2 - \beta +m}{2(\beta-1)}  \int_{z_1}^{z_2} \varphi_{z_0} (s) ds - \frac{\beta -m}{2(\beta-1)} \left(z_2  \frac{\eta_K}{2} - z_1 \eta_K\right).$$
Then, as $z_0 \to +\infty$, we get
$$(\varphi_{z_0}^m)' (z_2) - (\varphi_{z_0}^m)' (z_1) = z_1 \frac{\beta -m}{2 (\beta -1)} \frac{\eta_K}{2} +O(1) .$$
Since $z_1 \geq z_0$, clearly the right hand side goes to $+\infty$ as $z_0 \to +\infty$. Moreover, due to Proposition~\ref{prop:far-away}, we also know that $\varphi_{z_0} ' (z_1)$ has to remain bounded as $z_0$ is increased. Therefore, we find that $\varphi_{z_0} ' (z_2) \to +\infty$, which obviously is a contradiction.
\end{proof}
\begin{rem}\label{rem:up3}
The above proof relies on the fact that $\psi_{\infty,K}$ is a supersolution. Therefore in the regime $m+\frac 2 \alpha \leq \beta <2-m$, $\frac{1}{1-m}<\alpha<\frac{2}{1-m}$, we select some $\frac{2}{\beta-m}<\rho<\frac{2 \beta}{\beta-m}$ and  \eqref{eq:up} remains true when replacing the right hand term by 
$$\frac{A}{z^{\frac{2}{\beta-m}}} + \frac B{z^{\rho}},$$
where $A,B>0$, due to the fact it is also a supersolution (on a right half line). This means that, for any $z_0$ large enough, the function $\varphi_{z_0}$ already provides a self-similar solution as required in Remark~\ref{rem:up2} (notice that, in the critical case $m + \frac 2 \alpha = \beta$, one needs $A \geq \overline{C}$ and therefore the choice of $z_0$ depends on the initial data). This in turn justifies Remark~\ref{rem:up1}.
\end{rem}
We next investigate the decay of $\varphi _{z_0}$ when the blow-up point $z_0$ is small.

\begin{prop}[The decay from above for small blow-up points]\label{prop:fast-decay}
When $z_0>0$ is small enough, the solution $\varphi_{z_0}$ of the ODE \eqref{eq-self-similar} constructed in Proposition \ref{prop:blow-up} satisfies
$$
\varphi_{z_0} (z)\leq \psi_\infty(z)=\frac{C_\infty}{(z+Z)^{\frac{2}{1-m}}},
$$
on a right half line, where the shift $Z>0$ was selected in the beginning of subsection \ref{ss:beh-infinity}.
\end{prop}

\begin{proof} In view of Proposition \ref{prop:blow-up}, \eqref{enlarged}, \eqref{kappa} and \eqref{def:C0}, we see that
$$
\varphi_{z_0} ((1+\kappa)z_0 ) \leq \gamma \left(\frac{\beta-m}{2 \kappa r(\beta-1)^2}\right)^{\frac{1}{\beta-1}}<\psi_\infty(0).
$$
The last inequality follows from our choice of the shift $Z$, see \eqref{gamma1}. Hence, if  $z_0$ is small enough, this means that $\varphi_{z_0}$ crosses $\psi_\infty$ and, by Proposition~\ref{prop:ode_comp} (and if necessary by continuity of solutions w.r.t. boundary conditions, see Remark~\ref{rem:comp1}), the function $\varphi _{z_0}$ has to remain below $\psi_\infty$ on the right of this intersection. The proposition is proved.
\end{proof}

\subsection{A matching argument}

The above Propositions~\ref{prop:slow-decay} and~\ref{prop:fast-decay} lead us to introduce the real number
$$
z^* := \sup\, \{ z_0 > 0 : \varphi_{z_0} \leq \psi_\infty \ \mbox{ somewhere} \}.
$$
Indeed, we already know from Proposition \ref{prop:fast-decay} that the above set contains small enough $z_0$ and thus is not empty. On the other hand, if $\varphi_{z_0}$ ``crosses" $\psi_\infty$ somewhere then, by Proposition~\ref{prop:ode_comp}, we have that $\varphi_{z_0} \leq \psi_\infty$ on the right of this contact point, and the same conclusion holds if $\varphi_{z_0}$ \lq\lq touches with the same slope'' $\psi_\infty$ somewhere, in view of Remark \ref{rem:comp1}. In particular, when $z_0$ is large it would contradict the estimate from below of the decay in Proposition~\ref{prop:slow-decay}. In other words,  the above set is bounded from above. Thus $z^*$ does exist.

Our goal is now to show that $\varphi_{z^*}$ has the expected behaviour. As a matter of fact, we will only show that there exists a solution blowing up at $z^*$ with the wanted asymptotics both at its blow-up point and at infinity. The reason is that we lack uniqueness of the solution blowing up \lq\lq correctly'' at any $z_0$, and therefore we also lack continuity of $\varphi_{z_0}$ with respect to $z_0$. This makes the last part of the proof much harder.

\medskip

First, by the definition of $z^*$, we can find a sequence $z_n \searrow z^*$ such that $\varphi_{z_n}$ is above $\psi_\infty$. Passing to the limit as $n \to +\infty$ (by usual estimates), we find a new nonincreasing solution~$\varphi_{1,\infty}$. 
We know from Proposition \ref{prop:blow-up} that the interval of validity of estimate \eqref{validity} for $\varphi_{z_n}$ is $(z_n,\min(z_n+\delta(z_n),(1+\kappa)z_n))$ and that $\liminf _{n\to+\infty}\delta(z_n)>0$. As a result, we collect for $\varphi_{1,\infty}$ an asymptotics of the form
\begin{equation}\label{bup-wanted}
\frac{k_0}{(z-z^*)^{\frac{1}{\beta-1}}} \leq \varphi_{1,\infty} (z) \leq \frac{K_0}{(z-z^*)^{\frac{1}{\beta-1}}}
\end{equation}
on a neighborhood of $z^*$. As before, from the fact that $\varphi_{1,\infty}$ blows up, and since a solution of~\eqref{eq-self-similar} cannot reach a local minimum, one deduces that $\varphi_{1,\infty}$ is a decreasing function.

Moreover, by construction we have, for any $k<C_\infty$, 
$$
\varphi_{1,\infty} (z)\geq \psi_\infty(z)=\frac{C_\infty}{(z+Z)^{\frac{2}{1-m}}}\geq \frac{k}{z^{\frac{2}{1-m}}},
$$
where the first inequality holds for all $z > z^*$ and the second, which is the wanted lower estimate at infinity, holds for $z$ large enough. If, for some $K>C_\infty$, we have the wanted upper estimate, namely $\varphi_{1,\infty}(z)\leq \frac{K}{z^{\frac{2}{1-m}}}$ for $z$ large enough, then we are done. If not, then, $z\mapsto \frac{K}{z^{\frac{2}{1-m}}}$ being a supersolution, we see by another application of Proposition~\ref{prop:ode_comp} (see also Remark~\ref{rem:comp1}) that, for any $K>C_\infty$,
\begin{equation}\label{rate}
\varphi_{1,\infty}(z) \geq \frac{K}{z^{\frac{2}{1-m}}} \quad \text{ for all $z$ large enough}
\end{equation}
must hold.

Now take $\tilde{z}_n \nearrow z^*$ such that $\varphi_{\tilde{z}_n}$ touches $\psi_\infty$. As before, we can extract a converging subsequence to some $\varphi_{2,\infty}$, which is decreasing and also blows up at $z^*$ with the same asymptotics as in \eqref{bup-wanted}.

Next, observe that $\varphi_{1,\infty}$ is above $\varphi_{\tilde z_n}$ in neighborhoods of $z^*$ (where the former blows up) and of $+\infty$ (since we know that $\varphi_{1,\infty}> \psi_\infty$ from \eqref{rate} and $\psi _{\tilde z_n}\leq \psi_\infty$ as explained just after the definition of $z^*$ above). Hence, if  $\varphi_{\tilde{z}_n}$ and $\varphi_{1,\infty}$ intersect, they have to intersect twice which contradicts Proposition~\ref{prop:ode_comp} and continuity of solutions w.r.t. boundary conditions. As a result $\varphi_{\tilde z_n}\leq \varphi_{1,\infty}$ and thus, passing to the limit, we infer that
$$
\varphi_{2,\infty} \leq \varphi_{1,\infty}.
$$

\begin{cla}
If $\varphi_{2,\infty}$ lies above $\psi_\infty$, then it satisfies the wanted properties: for any $k<C_\infty$ and some $K> C_\infty$, we have that
$$
\frac{k}{z^{\frac{2}{1-m}}}\leq \varphi_{2,\infty} (z) \leq \frac{K}{z^{\frac{2}{1-m}}}\quad \text{ for $z$ large enough}.
$$
\end{cla}

\begin{proof} We only need to prove the upper bound. We proceed by contradiction: assume that for any $K>C_\infty$, there is a sequence $y_n \to +\infty$ such that $\varphi_{2,\infty}(y_n)>\frac{K}{y_n^{\frac{2}{1-m}}}= \psi_{\infty,K}(y_n)$. From Proposition~\ref{prop:ode_comp} and Remark~\ref{rem:comp1}, (recall that the $\psi_{\infty,K}$ are supersolutions), this implies that
\begin{equation}
\label{large-enough}
\varphi_{2,\infty} (z) \geq \frac{K}{z^{\frac{2}{1-m}}} \quad \text{ for $z$ large enough.}
\end{equation}

We shift $\psi_{\infty,K}$ to the right to get an intersection point, in the range where it is a supersolution. Since \eqref{large-enough} holds for any $K>C_\infty$,  $\varphi_{2,\infty}$ has to lie above any shift of $\psi_{\infty,K}$ at infinity. Therefore, we find a point $z_1$ where
$$
\varphi_{2,\infty}(z_1)=\psi_{\infty,K}(z_1-Z), \quad \varphi_{2,\infty} ' (z_1) \geq \psi_{\infty,K} ' (z_1 -Z).
$$
Here $Z$ is the shift chosen above. Due to the differential inequality satisfied by $\psi_{\infty,K}$, up to slightly changing $Z$ and $z_1$, the slope inequality is even a strict inequality. In particular, we find for $n$ large enough that there exists $z_{1,n}$ such that
$$\varphi_{\tilde{z}_n} (z_{1,n}) = \psi_{\infty,K} (z_{1,n} - Z), \quad \varphi_{\tilde{z}_n} ' (z_{1,n}) > \psi_{\infty,K}' (z_{1,n} - Z).$$
By Proposition~\ref{prop:ode_comp}, $\varphi_{\tilde{z}_n}$ has to lie above the shifted $\psi_{\infty,K}$ on a right half line, contradicting the fact that it crosses $\psi_\infty$ (and thus remains below it, also by Proposition~\ref{prop:ode_comp}). The claim is proved.
\end{proof}

We are left with the case when neither $\varphi_{1,\infty}$ nor $\varphi_{2,\infty}$ satisfy the wanted properties. According to the above arguments, this means that $\varphi_{2,\infty} \leq \not \equiv \psi_\infty$ on a right half-line, and that $\varphi_{1,\infty}$  has \lq\lq slow'' decay at infinity, that is \eqref{rate} holds. Moreover, $\varphi_{2,\infty} < \varphi_{1,\infty}$ by the strong maximum principle.

Another continuation approach is needed. Consider $K>C_\infty$ and a level set $\eta$, so that $\psi_{\infty,K}$ is a supersolution when less than $\eta$. Now, for $0\leq \theta \leq 1$, consider
$$ 
z_\theta:= (1-\theta) z_2 + \theta z_1,
$$
 where $z_2<z_1$ are the points where $\varphi_{2,\infty}$ and $\varphi_{1,\infty}$ respectively take the value $\eta$.

Next consider $\varphi_\xi$ the solution of the ODE \eqref{eq-self-similar} with boundary conditions
$$\varphi (z_\theta) = \eta , \quad \varphi ' (z_\theta) = - \xi.$$
For any $0\leq \theta\leq 1$, we claim that there exists a unique $\xi$ such that the solution $\varphi_\xi$ is decreasing and satisfies
$$\varphi_{2,\infty} \leq \varphi_{\xi} \leq \varphi_{1,\infty},$$
on $(z^*,z_\theta]$. The arguments for the existence of such $\xi$ are similar to those of subsection~\ref{ss:bup} (see in particular the proof of Claim~\ref{claim1}) and we only sketch them: as $\xi$ increases or decreases, the solution intersects $\varphi_{1,\infty}$ or $\varphi_{2,\infty}$; such an intersection has to be unique, meaning the solutions cannot intersect twice; take the largest $\xi$ so that the solution remains below $\varphi_{1,\infty}$, and one can then find that the solution has to remain between $\varphi_{1,\infty}$ and $\varphi_{2,\infty}$ as announced. Monotonicity follows as before from the impossibility of a local minimum. As for the uniqueness of such a $\xi$, it follows from the fact that two solutions cannot intersect and have the same blow-up point, see Proposition~\ref{prop:ode_comp_blowup}.

We denote by $\xi_\theta$ such a $\xi$. Thanks to the uniqueness, one can check that it is a continuous function of~$\theta$. Now take 
$$
\theta^* := \inf \{ \theta \in [0,1]: \varphi_{\xi_\theta} \geq \psi_\infty \}.
$$
Since $\varphi_{1,\infty}$ is above $\psi_\infty$, the above set contains $\theta=1$ and thus $\theta ^{*}$ is well defined. Since $\varphi_{2,\infty}$ crosses $\psi_\infty$, the above set does not contain $\theta =0$ and thus, by continuity with respect to $\theta$, we have that $\theta ^{*}>0$.

Finally, let us check that $\varphi_{\xi_{\theta^*}}$ satisfies all the wanted properties. There is no issue as far as blow-up is concerned, as it lies between $\varphi_{1,\infty}$ and $\varphi_{2,\infty}$. We also know, by continuity, that $\varphi_{\xi_{\theta^*}} \geq \psi_\infty$. Thus we only need to check that it does not have \lq\lq slow'' decay. Proceed by contradiction and assume that, for any $K>C_\infty$,
$$
\varphi_{\xi_{\theta^*}}(z) > \psi_{\infty,K}(z) \quad \text{ for $z$ large enough}.
$$
We claim that $\varphi_{\xi_{\theta^*}}$ is steeper than $\psi_{\infty,K}$, in the sense that its derivative is lower on any given (small, so that differential inequalities are available) level set. Precisely, select $\eta$ small enough so that $\psi_{\infty,K}$ is a supersolution when smaller than $\eta$. Select $z_2<z_1$ such that $\varphi_{\xi_{\theta^{*}}}(z_1)=\psi_{\infty,K}(z_2)=\eta$ and assume, by contradiction, that $\varphi_{\xi_{\theta^{*}}}'(z_1)>\psi_{\infty,K}'(z_2)$. By continuity we have, for a small $\ep>0$, the same equality and inequality between $\varphi _{\xi_{\theta ^{*} - \ep}}$ at a point $z_1^{\ep}$ and $\psi_{\infty,K}$ at point $z_2$. From Proposition~\ref{prop:ode_comp}, we infer that $\varphi_{\xi_{\theta ^{*} - \ep}}$ has to lie above the shifted $\psi_{\infty,K}(z-(z_1^{\ep}-z_2))$ for large~$z$, which contradicts the fact that it has to cross and then stay below $\psi _\infty$ for large $z$. Hence, as announced, $\varphi_{\xi_{\theta^{*}}}$ is steeper than $\psi_{\infty,K}$. 

From this steepness information, we deduce that $\varphi_{\xi_{\theta^{*}}}$ has to stay below another \lq\lq critical shift'' of the supersolution $\psi_{\infty,K}$, and thus below the (non shifted) $\psi_{\infty,K+1}$, which is a contradiction. We conclude that  $\varphi_{\xi_{\theta^*}}$ satisfies the wanted asymptotics.

\medskip

The proof is (almost) complete: it actually remains to estimate the derivatives of the constructed solution, which we now denote by $\varphi$, as $z \to +\infty$. We know that $\varphi$ solves \eqref{eq-self-similar} together with the estimate \eqref{self-infini} at $+\infty$. Since $z\varphi'(z)=(z\varphi(z))'-\varphi(z) $, equation \eqref{eq-self-similar} is recast
\begin{equation}\label{manip}
-(\varphi ^{m})''(z)=c_1 \varphi(z)+r\varphi^{\beta}(z)+c_2(z\varphi(z))'
\end{equation}
where $c_1:=\frac{2-\beta+m}{2(\beta-1)}>0$, $c_2:=\frac{\beta-m}{2(\beta-1)}>0$. From estimate \eqref{self-infini}, the first two terms in the right hand side are integrable at $+\infty$, and $z\varphi(z)\to 0$ as $+\infty$. As a result,  $(\varphi^{m})'(z)$ must have a finite limit as $z\to+\infty$ which must be zero (if not, since $\varphi$ is decreasing, then $\varphi^{m}$ becomes negative at infinity). Hence we get
$$
0 > m\varphi ^{m-1}(z)\varphi'(z)=(\varphi^{m})'(z)=\int _z^{+\infty}(c_1\varphi(s)+r\varphi^{\beta}(s))ds-c_2z\varphi(z)\geq -c_2z\varphi(z).
$$
Using again \eqref{self-infini} we see that the above yields the estimate on $\varphi'$ in \eqref{to-be-proved}, whereas the estimate on $\varphi''$ then directly follows from the ODE \eqref{eq-self-similar} and the estimates on $\varphi$ and $\varphi '$. \qed

\bigskip

\noindent{\bf Acknowledgements.} M. Alfaro is supported by the 
ANR I-SITE MUSE, project MICHEL 170544IA (n$^{o}$ ANR-IDEX-0006). T. Giletti is supported by the NONLOCAL project (n$^{o}$ ANR-14-CE25-0013).

\bibliographystyle{siam}    
\bibliography{biblio2}

 \end{document}